\documentclass[12pt]{amsart}
\usepackage[all]{xy}
\usepackage{biblatex}
\usepackage{amssymb}
\usepackage{amsmath}
\usepackage{mathrsfs}
\usepackage{amsfonts}
\usepackage{amsmath,amsfonts,amssymb}
\usepackage{cases}
\usepackage{indentfirst}
\usepackage{ifthen}
\usepackage{amsthm}
\usepackage{graphicx}
\usepackage{epstopdf}
\usepackage{changepage}
\usepackage{geometry}
\usepackage{bbm}
\usepackage{xcolor}
\usepackage{hyperref}
\usepackage{tikz}

\numberwithin{equation}{section}

\newcommand{\RR}{\mathbb R}
\newcommand{\CC}{\mathbb C}

\newcommand{\ZZ}{\mathbb Z}
\newcommand{\QQ}{\mathbb Q}

\newcommand{\bM}{\mathbf{M}}
\newcommand{\bN}{\mathbf{N}}
\newcommand{\bP}{\mathbf P}
\newcommand{\bG}{\mathbf G}

\newcommand{\bB}{\mathbf B}

\newcommand{\bU}{\mathbf U}
\newcommand{\bA}{\mathbf A}
\newcommand{\bH}{\mathbf H}
\newcommand{\bT}{\mathbf T}

\newcommand{\cO}{\mathcal O}

\newcommand{\cG}{\mathcal G}

\newcommand{\cS}{\mathcal S}

\newcommand{\fa}{\mathfrak a}
\newcommand{\fo}{\mathfrak o}

\newcommand{\fp}{\mathfrak p}
\newcommand{\fw}{\mathfrak w}

\newcommand{\rG}{\mathrm G}

\newcommand{\SL}{\mathrm{SL}}

\newcommand{\GL}{\mathrm{GL}}

\newcommand{\SO}{\mathrm{SO}}

\newcommand{\SU}{\mathrm{SU}}

\newcommand{\lieg}{\mathfrak g}

\newcommand{\liem}{\mathfrak m}
\newcommand{\lien}{\mathfrak n}

\newcommand{\liez}{\mathfrak z}

\newcommand{\Lie}{\mathrm{Lie}}
\newcommand{\Sym}{\mathrm{Sym}}

\newcommand{\Spec}{\mathrm{Spec}}
\newcommand{\Hom}{\mathrm{Hom}}

\newcommand{\Res}{\mathrm{Res}}
\newcommand{\Frob}{\mathrm{Frob}}

\newcommand{\der}{\mathrm{der}}
\newcommand{\Ad}{\mathrm{Ad}}
\newcommand{\ad}{\mathrm{ad}}
\newcommand{\rank}{\mathrm{rank}}
\newcommand{\WD}{\mathrm{WD}}
\newcommand{\W}{\mathrm{W}}
\newcommand{\Sh}{\mathrm{Sh}}
\newcommand{\red}{\mathrm{red}}
\newcommand{\res}{\mathrm{res}}
\newcommand{\ur}{\mathrm{ur}}
\newcommand{\vol}{\mathrm{vol}}
\newcommand{\std}{\mathrm{std}}
\newcommand{\Gal}{\mathrm{Gal}}
\newcommand{\Irr}{\mathrm{Irr}}

\newcommand{\isom}{\simeq}

\newtheorem{theorem}{Theorem}[section]
\newtheorem{proposition}[theorem]{Proposition}
\newtheorem{lemma}[theorem]{Lemma}
\newtheorem{corollary}[theorem]{Corollary}
\newtheorem{conjecture}[theorem]{Conjecture}

\newtheorem{fact}[theorem]{Fact}
\newtheorem*{theorem*}{Theorem}
\newtheorem*{example}{Example}

\newtheorem*{remark}{Remark}

\begin{document}

\title{Formal Degrees and Parabolic Induction: the Maximal Generic Case}
\author{ Yiyang Wang}
\date{}
\address{
Department of Mathematics, Kyoto University, Kitashirakawa Oiwake-cho, Sakyo-ku, Kyoto 606-8502,
Japan
}
\email{wang.yiyang.67u@st.kyoto-u.ac.jp}
\maketitle
\begin{abstract}
   We study the compatibility of the formal degree conjecture and the parabolic induction process in the simplest nontrivial case for quasi-split $p$-adic groups. For a generic discrete series $\pi$ induced from an irreducible supercuspidal $\sigma$ of a maximal Levi subgroup, we compute the quotient $d(\pi)/d(\sigma)$ of formal degrees with the assumption that the group is unramified. As an application, we verify the conjecture for discrete series of split $\rG_2$ supported on maximal Levi subgroups.
\end{abstract}
\addtocontents{toc}{\setcounter{tocdepth}{1}} 
\tableofcontents

\section{Introduction}
Let $F$ be a non-archimedean local field of characteristic $0$ and residual characteristic $p$. Let $\bG$ be a connected reductive group over $F$ with $\bA$ the split component of its center. Set $G=\bG(F)$ and $A=\bA(F)$.

Let $(\pi,V_\pi)$ be a discrete series (i.e. irreducible unitary, square-integrable modulo $A$) of $G$ with a fixed invariant inner product $(\cdot,\cdot)$ on $V_\pi$. Let $\mu$ be a fixed Haar measure on the quotient group $G/A$. Recall that the formal degree of $\pi$, with respect to the choice of the measure $\mu$, is defined to be the unique positive real number $d(\pi)=d(\pi,\mu)\in\RR_{>0}$ such that
\[\int_{G/A}(\pi(g)v,v')\overline{(\pi(g)w,w')}\mu(\dot{g})=\frac{1}{d(\pi)}(v,w)\overline{(v',w')},\quad v,v',w,w'\in V_\pi.\]

The formal degree of a discrete series of $G$ is exactly its Plancherel measure in Harish-Chandra's Plancherel formula. For real reductive groups, Harish-Chandra exhausted the discrete series and established the explicit Plancherel formula in 1970s. Following Harish-Chandra, Waldspurger completed the proof of the Plancherel formula for $p$-adic groups in \cite{waldspurger-plancherel}, but the Plancherel measure could not be explicitly determined since there's no explicit classification of even supercuspidals for general $p$-adic groups. It was in \cite{hiraga-ichino-ikeda}\cite{hiraga-ichino-ikeda-correction}, assuming the local Langlands correspondence (conjectural for general $\bG$), that an explicit formula for $d(\pi)$ was conjectured in terms of the adjoint $\gamma$-factor. 

More precisely, let $\psi$ be a fixed non-trivial additive character of $F$ and $\mu_\psi(g)$ a specific Haar measure on $G$ depending on $\psi$ (cf. \textsection \ref{section measures}). Assuming the local Langlands correspondence for $\bG$, let $(\varphi_\pi,\rho_\pi)$ be the refined Langlands parameter of $\pi$. Thus $\varphi_\pi:\WD_F\rightarrow {^LG}$ is an admissible homomorphism from the Weil-Deligne group $\WD_F=\W_F\times \SL_2(\CC)$ into the $L$-group ${^LG}$, and $\rho_\pi$ is an irreducible character of a certain component group (cf. \S\ref{subsection the conjecture}). Let $\widehat{G}$ (resp. $\widehat{G}^\natural$) be the dual group of $\bG$ (resp. $\bG/\bA$), so that $\widehat{G}^\natural$ is a subgroup of $\widehat{G}$. Let $S_{\varphi_\pi}^\natural:=Z_{\widehat{G}^\natural}(\varphi_\pi(\WD_F))$, and $\mathcal{S}_{\varphi_\pi}^\natural:=\pi_0(S_{\varphi_\pi}^\natural)$ be the corresponding component group. It was conjectured that 
$$d(\pi)=\frac{\dim \rho_\pi}{|\mathcal{S}^\natural_{\varphi_\pi}|}\cdot |\gamma(0,\pi,\Ad,\psi)|,$$
where $\Ad$ is the adjoint representation of ${^LG}$, and $\gamma(s,\pi,\Ad,\psi)$ the corresponding adjoint $\gamma$-factor.

By general facts in representation theory of $p$-adic groups, every discrete series $\pi$ can be parabolically induced from some irreducible supercuspidal representation of a Levi subgroup. This means
\begin{itemize}
    \item there exists an $F$-parabolic subgroup $\bP$ with Levi decomposition $\bP=\bM\bN$, an irreducible unitary supercuspidal representation $\sigma$ of $M$, and some $\lambda\in \fa^*_{\bM}=X^*(\bM)\otimes \RR$ such that $\pi$ is a subquotient of $i_P^G\sigma_\lambda$ ($\sigma_\lambda:=\sigma\otimes \chi_\lambda$, where $\chi$ is the unramified character of $M$ associated to $\lambda\in \fa^*_{\bM}$, cf. \textsection \ref{section preliminaries}).
\end{itemize} Then in order to reduce the formal degree conjecture to the supercuspidal case, it is natural to compute $\frac{d(\pi)}{d(\sigma)}$ and try to verify 
\begin{equation}\label{formal degree quotient}
    \frac{d(\pi)}{d(\sigma)}=\frac{\dim \rho_\pi}{\dim\rho_\sigma}\cdot \left(\frac{|\mathcal{S}^\natural_{\varphi_\pi}|}{|\mathcal{S}^\natural_{\varphi_\sigma}|}\right)^{-1}\cdot \frac{|\gamma(0,\pi,\Ad,\psi)|}{|\gamma(0,\sigma,\Ad,\psi)|},
\end{equation}
where $d(\sigma),\varphi_\sigma, \rho_\sigma,\mathcal{S}_{\varphi_\sigma}^\natural$ (resp. $d(\pi),\varphi_\pi, \rho_\pi,\mathcal{S}_{\varphi_\pi}^\natural$) are the corresponding objects attached to $\sigma$ (resp. $\pi$).

The main purpose of this paper is to study certain special cases of \eqref{formal degree quotient} using tools from local harmonic analysis. The ideas have already been used to deal with certain examples of classical groups in \cite{gan-ichino-formal-degree}, while there are several technical difficulties for general cases:
\begin{itemize}
\item  First, the (refined) local Langlands correspondence (LLC) is currently inaccessible for general $p$-adic groups.
    \item Secondly, the Langlands-Shahidi method works only for generic representations of quasi-split groups.
    \item Finally, for general (not necessarily maximal) parabolic subgroups, the multiplicity of $\pi$ ``appearing" in the subquotients of $i_P^G\sigma_{\lambda}$ is quite subtle.
\end{itemize} 

In this paper, we shall deal with the simplest nontrivial case. Let $\bG$ be a quasi-split connected reductive group over $F$ with fixed Borel subgroup, and $\bP=\bM\bN$ standard maximal parabolic subgroup with $\Tilde{\alpha}\in \fa_{\bM}^*$ the corresponding fundamental weight (cf. \S\ref{section harmonic analysis} for their precise meanings). 

Let $\pi$ be a generic discrete series representation of $G$ supported on $\bM$, then by Shahidi's classification there exists an irreducible unitary supercuspidal representation $\sigma$ of $M$ and a unique $j\in\{1,2\}$ such that $\pi$ is the unique generic subrepresentation of $i_P^G\sigma_{\Tilde{\alpha}/j}$. 

Finally, to make the final explicit formula less technical, we assume further that $\bG$ is unramified, i.e. splits over a finite unramified extension of $F$. (This condition is  inessential for our approach, but would make the quotient of two different measures simpler, cf. \S \ref{section measures}). With these settings, the main theorem (Theorem \ref{main theorem}) of this paper reads:
\begin{theorem*}
    In the above setting, assuming the (refined) local Langlands correspondence for $\bG$ and $\bM$ (conjectural in general) with some natural assumptions, we have
    \[
    \frac{d(\pi)}{d(\sigma)}=j^{-1}\cdot \frac{m}{\langle \chi,\alpha^\vee\rangle} \cdot  \frac{|\gamma(0,\pi,\Ad,\psi)|}{|\gamma(0,\sigma,\Ad,\psi)|},
    \]
    where 
    \begin{itemize}
        \item $\chi\in X^*(\mathbf{M})^{\mathbf{G}}\isom \ZZ$ is the generator with $\langle \chi,\alpha^\vee\rangle>0$;
        \item $m$ is the index of $\res(X^*(\mathbf{M})^{\mathbf{G}})$ in $X^*(\mathbf{A}_{\mathbf{M}})^{\mathbf{G}}$ (cf. \S\ref{section harmonic analysis} for the notations).
    \end{itemize}
    
    In particular, this result is compatible with the formal degree conjecture if and only if
\[
\frac{|\mathcal{S}^\natural_{\varphi_\pi}|}{|\mathcal{S}^\natural_{\varphi_\sigma}|}=j\cdot \frac{\langle \chi,\alpha^\vee\rangle}{m}.
\]
\end{theorem*}
This also shows how Langlands parameters could reflect information from local harmonic analysis and representation theory. As an application, we apply this to prove the formal degree conjecture for discrete series of split $\mathrm{G}_2$ supported on a maximal Levi (based on the local Langlands for $\mathrm{G}_2$ established recently by Gan-Savin \cite{gan-savin-g2theta}\cite{gan-savin-g2-llc}).

Our computation is a combination of Heiermann's formula for the quotient $d(\pi)/d(\sigma)$ (in terms of the residue of Harish-Chandra's Plancherel density), and the Langlands-Shahidi method. The main difficulty is the explicit determination of some constants, which is somewhat technical but natural in principle. We also remark that for general non-generic representations, besides the invalidity of the Langlands-Shahidi method, the complexity of the behaviour of parameters also serves as an essential obstacle to the explicit computation of the adjoint $\gamma$-factor $\gamma(s,\pi,\Ad,\psi)$ in this approach.

It remains to briefly explain the content of this paper.  
We first recollect some basic facts and fix the notations in \textsection \ref{section preliminaries}. In \textsection\ref{section measures}, we deal with the technical comparison of two different Haar measures. In \textsection \ref{section harmonic analysis} and \textsection\ref{section parameters}, we make necessary preparations in local harmonic analysis, parameters, and compute the adjoint $\gamma$-factors. We then recall the formal degree conjecture and prove the main theorem in \textsection\ref{section the main theorem}. Finally, as an application, we apply this result to discrete series of split $\mathrm{G}_2$ supported on a maximal Levi in \textsection \ref{section G2}.

\section{Preliminaries}\label{section preliminaries}
\subsection{General notations}\label{subsection general notations}
The following is a list of notations that will be used throughout this paper:
\begin{itemize}
    \item $F$ is a non-archimedean local field of characteristic $0$, with the ring of integers $\fo=\fo_F$, its maximal ideal $\fp=\fp_F$, and a fixed uniformizer $\varpi\in \fp$. The order of the residue field $\fo/\fp$ is denoted by $q$. We fix an algebraic closure $\overline{F}$ of $F$.
    \item  Denote by $|\cdot |=|\cdot |_F$ the normalized absolute value of $F$, so that $|\varpi|:=q^{-1}$. We also fix a nontrivial additive (unitary) character $\psi:F\rightarrow \CC^1$. This normalizes the definition of the Fourier transform on $F$. We then equip $F$ with the corresponding self-dual measure, characterized by the Fourier inversion.
    \item $\W_F$ is the Weil group of $F$, with the inertia subgroup $\mathrm{I}_F$, and $\WD_F:=\W_F\times \SL_2(\CC)$ is the Weil-Deligne group of $F$, cf. \S\ref{subsection parameters} for more details.
     \item Given an algebraic group $\bH$ (not necessarily connected) over $F$ or $\CC$, we denote by $\bH^\circ$ its identity component, $\bA_{\bH}$ the split component of its center. Denote by $X^*(\bH)$ the (free abelian) group of $F$-rational characters of $\bH$, and $\mathfrak{a}_{\bH}^*$ (resp. $\mathfrak{a}_{\bH,\CC}^*$) the real (resp. complex) vector space $X^*(\bH)\otimes\RR$ (resp. $X^*(\bH)\otimes\CC$).

    \item  We shall use boldfaced letters to denote algebraic groups (group schemes) and usual letters are used for the group of $F$-points, e.g., $H=\bH(F)$, $A_{\bH}=\bA_{\bH}(F)$.

      \item By ``a parabolic subgroup $\bP=\bM\bN$" of some reductive group, we mean an $F$-parabolic subgroup $\bP$ with a Levi decomposition, $\bN$ its unipotent radical and $\bM$ a Levi subgroup.
      \item All representations for $p$-adic groups considered in this paper are smooth representations over $\CC$.
      \end{itemize}

For a quasi-split connected reductive group $\bG$ with a fixed Borel subgroup $\bB=\bT\bU$, we shall use the following standard facts, notations and terminology.
      \begin{itemize}
    \item   Denote by $\bT_0$ the split component of $\bT$, which is a maximal $F$-split torus of $\bG$, and $X_*(\bT_0)$ the set of cocharacters of $\bT_0$. 
    \item Denote by $\Sigma=\Sigma(\bG,\bT_0)\subset X^*(\bT_0)$ (resp. $\Sigma^\vee=\Sigma^\vee(\bG,\bT_0)\subset X_*(\bT_0)$) the set of relative roots (resp. coroots) of $\bG$ with respect to $\bT_0$, with the correspondence $\alpha\mapsto \alpha^\vee$, $\Sigma\rightarrow \Sigma^\vee$. 
    \item The choice of the Borel subgroup $\bB$ determines a subset $\Sigma^+$ (resp. $\Delta$) of positive roots (resp. positive simple roots) of $\Sigma$. We denote by $\Sigma_{\red}$ the set of reduced roots. 
    \item A parabolic subgroup $\bP=\bM\bN$ of $\bG$ is said to be standard, if $\bP\supset \bB$, $\bM\supset \bT$. In this case we denote by $\bar{\bP}=\bM\bar{\bN}$ its opposite parabolic subgroup, and denote by $\Sigma(\bP)$ (resp. $\Sigma_{\red}(\bP)$) the subset of roots (resp. reduced roots) appearing in the Lie algebra of $\bN$.
    \item The map 
    \[
    \bP\mapsto \Delta-\Sigma(\bP) ,~\{\text{standard parabolic subgroups of }\bG\}\rightarrow \{\text{subsets of }\Delta\}
    \]
    is an inclusion preserving bijection. In particular, a maximal standard parabolic subgroup corresponds to a maximal subset $\theta=\Delta-\{\alpha\}$ for a unique $\alpha\in \Delta$.
\end{itemize}

\subsection{Harmonic analysis}\label{subsection preliminaries harmonic}
In this subsection we recollect some basic notions and facts frequently used in local harmonic analysis. A standard reference is \cite{waldspurger-plancherel}. Let $\bH$ be an arbitrary connected reductive group over $F$ with $\bA_{\bH}$ the split component of its center.
    \begin{itemize}
        \item For any $\chi\in X^*(\bH)$, $s\in \CC$, we denote by $|\chi|^s$ the unramified character
\[
|\chi|^s:H\rightarrow \CC^\times,~h\mapsto |\chi(h)|^s.
\]

\item Let $H^1:=\cap_{\chi\in X^*(\bH)}\ker |\chi|$ and 
\[
X^{\ur}(H):=\Hom(H/H^1,\CC^\times)
\]
be the group of unramified characters of $H$. We thus have a surjection $\mathfrak{a}_{\bH,\CC}^*\twoheadrightarrow X^{\ur}(H)$ characterized by $\chi\otimes s\mapsto |\chi|^s$, which equips $X^{\ur}(H)$ with a canonical complex structure. For $\lambda\in \mathfrak{a}_{\bH,\CC}^*$, we denote by $\chi_\lambda\in X^{\ur}(H)$ its image. 
\item To simplify the notations, given a smooth representation $\sigma$ of $H=\bH(F)$ and $\lambda\in \mathfrak{a}_{\bH,\CC}^*$, we also denote by $\sigma_{\lambda}$ the representation $\sigma\otimes \chi_{\lambda}$.
\item Denote by 
\[
X_0^{\ur}(H):=\{\chi\in X^{\ur}(H)\mid \chi(H)\subset \CC^{1}\}\subset X^{\ur}(H)
\]
the group of unitary unramified characters, i.e. the image of $i\mathfrak{a}_{\bH}^*$ under $\mathfrak{a}_{\bH,\CC}^*\twoheadrightarrow X^{\ur}(H)$.
\item For convenience, an irreducible smooth representation $\pi$ of $H=\bH(F)$ will be called a discrete series (representation) of $H$ if its central character is unitary and $\pi$ is square integrable modulo $A_{\bH}$. (Although for non-semisimple $\bH$ ``essentially square integrable representation" would be a more precise terminology.)

\end{itemize}

Now let $\bG$ be a connected reductive group over $F$ with a fixed maximal $F$-split torus $\bT_0$. Let $\bP=\bM\bN$ be a parabolic subgroup with $\bM\supset \bT_0$. We write $\fa_0^*=\fa_{\bT_0}^*$.
\begin{itemize}
    \item The inclusion $\bA_{\bG}\subset \bG$ induces the restriction morphism
    \[
    \res:X^*(\bG)\rightarrow X^*(\bA_{\bG})
    \]
    and gives an isomorphism $\res\otimes\mathbbm{1}:\mathfrak{a}_{\bG}^*\xrightarrow{\sim} \mathfrak{a}_{\bA_{\bG}}^*$ of vectors spaces after $(\cdot )\otimes \RR$.
    \item The inclusions $\bA_{\bM}\subset \bT_0\subset \bM$ induces 
    \[
    X^*(\bM)\xrightarrow{\res}X^*(\bT_0)\xrightarrow{\res}X^*(\bA_{\bM})
    \]
    and gives a canonical decomposition 
    \[
    \mathfrak{a}_{0}^*=\mathfrak{a}_{\bM}^*\oplus \mathfrak{a}_0^{\bM *}.
    \]
    (Also cf. \cite[\textsection I.1]{waldspurger-plancherel}.) Similarly, $\bA_{\bG}\subset \bA_{\bM}\subset \bM\subset \bG$ induces a canonical decomposition
    \[
   \mathfrak{a}_{\bM}^* =\mathfrak{a}_{\bG}^*\oplus\mathfrak{a}_{\bM}^{\bG *}.
    \]
    It would be convenient to identify $\mathfrak{a}_{\bM}^{\bG *}$ with $X^*(\bA_{\bM}/\bA_{\bG})\otimes \RR$.
    \item Given an irreducible representation $\sigma$ of $M$, denote by
    \[
    \cO=\cO_{\sigma}:=\{\sigma\otimes \chi \mid \chi\in X^{\ur}(M)\}\quad\text{(in the sense of isomorphism classes)}
    \]
    the orbit of $\sigma$. It is equipped with a canonical complex structure through
    \[
    \fa_{\bM,\CC}^*\twoheadrightarrow X^{\ur}(M)\twoheadrightarrow \cO,\quad \lambda\mapsto \chi_{\lambda}\mapsto (\sigma\otimes \chi_{\lambda}).
    \]
    We shall also need a relative version of the orbit, cf. \S\ref{subsection Heiermann's formula}.
\end{itemize}
Finally, for the convenience of the reader, we recall the structure constant $\gamma(\bG/\bM)$. Following the above setting, we need the following objects and notations.
\begin{itemize}
    \item We fix a special maximal compact subgroup $K_G$ of $G$ with $G=PK_G$. For any $g\in G$, fix a choice of $m_{\bP}(g)\in M$, $n_{\bP}(g)\in N$, $k_{\bP}(g)\in K_G$ such that $g=m_{\bP}(g) n_{\bP}(g) k_{\bP}(g)$.
    \item For any closed subgroup $H$ of $G$, we equip $H$ with the Haar measure such that $K_H:=K_G\cap H$ has volume one. This is the standard choice of measures in local harmonic analysis.
    \item Recall $\Bar{\bP}=\bM\Bar{\bN}$ is the opposite parabolic of $\bP$ (with respect to the choice of the maximal split torus $\bT_0$ and a Borel $\bB$). Let $\delta_{\bP}$ be the modulus character of $\bP$.
\end{itemize}    
Then the constant $\gamma(\bG/\bM)$ is defined to be
\[
\gamma(\bG/\bM)=\int_{\Bar{N}}\delta_{\bP}(m_{\bP}(\Bar{n}))d\Bar{n}.
\]
We remark that this $\gamma(\bG/\bM)$ is independent of the choice of $\bP$ containing $\bM$ (cf. \cite[pp. 240-241]{waldspurger-plancherel}), but does depend on the choice of $K_G$ (since the measure does).

\subsection{Parameters and local factors}\label{subsection parameters}
We recollect in this subsection some basic facts about Langlands parameters and local factors. For our purposes, stating them for quasi-split groups would suffice. Let $\bG$ be a quasi-split connected reductive group over $F$ with a fixed Borel subgroup $\bB=\bT\bU$.
\begin{itemize}
    \item Denote by $\widehat{G}$ the dual group of $\bG$. Recall that $\widehat{G}$ is then a connected reductive complex Lie group, whose root datum is the dual of the absolute root datum of $\bG$.
    \item After fixing a pinning of the based (absolute) root datum, $\Gamma:=\Gal(\overline{F}/F)$ acts on $\widehat{G}$ via ``pinned automorphisms". The L-group of $\bG$ is 
    \[
    {^LG}:=\widehat{G}\rtimes \Gamma.
    \]
    If $K/F$ is a finite Galois extension over which $G$ splits, then the action of $\Gamma$ factors through $\Gamma_{K/F}:=\Gal(K/F)$, and ${^LG}$ is essentially $\widehat{G}\rtimes \Gamma_{K/F}$.
    \item 
   We fix a Borel subgroup $\widehat{B}$ of $\widehat{G}$ corresponding to the based root datum. A parabolic subgroup ${^L\!P}$ of ${^LG}$ is the normalizer $N_{{^LG}}(\widehat{P})$ in ${^LG}$ of some parabolic $\widehat{P}=\widehat{M}\widehat{N}$ of $\widehat{G}$, with unipotent radical ${^L\!N}=\widehat{N}$ and Levi ${^L\!M}=N_{^L\!P}(\widehat{M})$. A parabolic ${^L\!P}$ of ${^LG}$ is said to be standard if ${^L\!P}\supset{^L\!B}$, which must be of the form ${^L\!P}=\widehat{P}\rtimes \Gamma$, $\widehat{P}\supset \widehat{B}$ corresponding to a $\Gamma$-invariant subset of positive simple roots, with ${^L\!M}=\widehat{M}\rtimes \Gamma$.  
   \item A standard parabolic subgroup $\bP=\bM\bN$ of $\bG$ then corresponds to a standard parabolic ${^L\!P}$ of ${^LG}$, which is relevant since $\bG$ is assumed to be quasi-split (\cite[\S 3]{borel-automorphic-L-function}). We could then identify the L-group of $\bM$ with ${^L\!M}$. For more details, cf. \cite{borel-automorphic-L-function}. We denote by $\widehat{\mathfrak{m}}$ (resp. $\widehat{\mathfrak{n}}$) the Lie algebra of ${\widehat{M}}$ (resp. ${\widehat{N}}$).
    \item An L-parameter of $\bG$ is a homomorphism $\varphi:\WD_F\rightarrow {^LG}$ such that
    \begin{itemize}
        \item $\varphi:\SL_2(\CC)\rightarrow \widehat{G}$ is algebraic;
        \item $\varphi$ is continuous on $\mathrm{I}_F$ and $\varphi(\Frob)\in \widehat{G}$ is semisimple;
        \item the composite $\W_F\xrightarrow{\varphi}{^LG}\twoheadrightarrow \Gamma$ is the natural inclusion $\W_F\hookrightarrow \Gamma$.
    \end{itemize}
    \item  Denote by $\Phi(\bG)$ (resp. $\Pi(\bG)$) the set of $\widehat{G}$-conjugacy classes of L-parameters of $\bG$ (resp. isomorphism classes of irreducible smooth representations of $G=\bG(F)$).
\end{itemize}

The local Langlands correspondence (abbreviated as ``LLC" from now on) for $\bG$, which is conjectural in general, is a finite-to-one surjection 
\[
\mathrm{LL}:\Pi(\bG)\twoheadrightarrow \Phi(\bG),\quad \pi\mapsto \varphi_{\pi},
\]
satisfying a series of properties. For any $\pi\in \Pi(\bG)$ (resp. $\varphi\in \Phi(\bG)$), $\mathrm{LL}(\pi):=\varphi_{\pi}$ (resp. $\mathrm{LL}^{-1}(\varphi):=\Pi_{\varphi}$) is called the L-parameter of $\pi$ (resp. the L-packet of $\varphi$).

Recall that we have fixed a nontrivial unitary character $\psi:F\rightarrow \CC^1$. Assuming LLC and given a finite dimensional representation $r$ of the complex Lie group ${^LG}$, we denote by
\[
L(s,\pi,r),~\varepsilon(s,\pi,r,\psi),~\gamma(s,\pi,r,\psi)
\]
the local $L,\varepsilon,\gamma$-factors of $\pi\in \Pi(\bG)$. They are those local factors of $r\circ \varphi$ as a representation of the Weil-Deligne group.

Finally, we fix a geometric Frobenius $\Frob\in \W_F$ and a reciprocity homomorphism
\[
\mathrm{rec}:\W_F\twoheadrightarrow F^\times
\]
by the local class field theory, which maps $\Frob$ to the uniformizer $\varpi$. We then write
\[
\|w\|:=|{\rm rec}(w)|_F,~v(w):=-\log_q\|w\|_F\quad (w\in \W_F).
\]

The choice of ${\rm rec}$ normalizes the LLC for split tori: for $\bG=\GL_1$, a character $\chi:F^\times \rightarrow \CC^\times$ corresponds to the parameter $\chi\circ \mathrm{rec}:\W_F\rightarrow \CC^\times$. For example, the unramified character $\chi=|\cdot |^{\lambda}~(\lambda\in \CC)$ corresponds to the unramified character $\W_F\rightarrow \CC^\times$ which maps $\Frob$ to $q^{-\lambda}$. We remark that in some literature such as \cite[\textsection 2.5]{heiermann-orbites}, the authors use a normalization which gives $q^{\lambda}$, not $q^{-\lambda}$.

\section{Haar Measures}\label{section measures}
Fix a Haar measure $\mu$ on $G/A$. Let $(\pi,V_\pi)$ be a discrete series of $G$ and $(\cdot,\cdot)$ a $G$-invariant inner product on $V_\pi$. Recall that by the Schur orthogonality relation of matrix coefficients, there exists a unique positive real number $d(\pi)=d(\pi,\mu)\in\RR^\times_{>0}$, called the formal degree of $\pi$ with respect to $\mu$, such that
\[\int_{G/A}(\pi(g)v,v')\overline{(\pi(g)w,w')}\mu(g)=\frac{1}{d(\pi)}(v,w)\overline{(v',w')},\quad v,v',w,w'\in V_\pi.\]
Clearly, $d(\pi)=d(\pi,\mu)$ is independent of the choice of $(\cdot,\cdot)$ but depends on the choice of $\mu$, and $d(\pi,\lambda\mu)=\lambda^{-1}\cdot d(\pi,\mu)$ for $\lambda\in \RR^\times_{>0}$.

We fix a special maximal compact subgroup $K_G\subset G=\bG(F)$, and for any closed subgroup $H$ of $G$, we write $K_{H}:=K_G\cap H$. In local harmonic analysis, one chooses the (left) Haar measure on any closed subgroup $H$ of $G$ such that $K_H$ has volume one. Thus if $H$ is a unimodular subgroup of $G$, $G/H$ is equipped with the left $G$-invariant quotient measure $d\dot g$, characterized by
$$\int_{G/H}\int_H f(gh)dh d\dot g=\int_G f(g)dg,\quad f\in C_c^\infty(G).$$
For $H=A_{\bG}$, we denote this quotient Haar measure on $G/A_{\bG}$ by $\mu_{1,\bG}$.

On the other hand, the formal degree conjecture uses another Haar measure specified as follows. We have fixed a non-trivial additive character $\psi$ of $F$. Let $\mathcal{G}$ be a split form of $\mathbf{G}/\mathbf{A}_{\bG}$ defined over $\fo=\fo_F$. Choose an $\overline{F}$-isomorphism
$$\eta: \mathbf{G}/\mathbf{A}_{\bG}\xrightarrow{\sim} \cG$$
and an invariant top form $\omega_{\cG}$ on $\cG$ defined over $\fo$ (which is equivalent to choosing a top form over $\fo$ on the Lie algebra of $\cG$), with nonzero reduction (modulo $\varpi=\varpi_F$). Clearly, $\omega_{\cG}$ is unique up to a scalar in $\fo^\times$. Set $\omega_{\bG}:=\eta^{*}(\omega_{\cG})$ and let $dg:=dg_\psi$ be the Haar measure on $G/A_{\bG}$ determined by $\omega_{\bG}$ and the self-dual measure on $F$ with respect to $\psi$ (cf. \S\ref{subsection general notations}). Then $dg$ depends only on $\psi$ and is independent of the choice of $\eta$ and $\omega_{\cG}$ (cf. \cite{gross-motive}\cite{gan-gross-haar-measure}). We denote this Haar measure on $G/A_{\bG}$ by $\mu_{2,\bG}$.

Now let $\bP=\bM\bN$ be a parabolic subgroup of $\bG$. The measures $\mu_{1,\bM}$, $\mu_{2,\bM}$ (on $M/A_{\bM}$) are defined similarly. Denote by $\deg(\cdot)$ (resp. $d(\cdot)$) the formal degree of a discrete series representation (of either $\bM$ or $\bG$) with respect to the Haar measure $\mu_1$ (resp. $\mu_2$). For convenience, we denote by $\frac{\mu_{1,\bG}}{\mu_{2,\bG}}$ (resp. $\frac{\mu_{1,\bM}}{\mu_{2,\bM}}$) the unique positive real number $\lambda$ such that $\mu_{1,\bG}=\lambda\mu_{2,\bG}$ (resp. $\mu_{1,\bM}=\lambda\mu_{2,\bM}$).

Let $\pi$ (resp. $\sigma$) be a discrete series representation of $G$ (resp. $M$). In order to use tools from harmonic analysis to deal with our problem, we have to first evaluate the quotient
\[
\left(\frac{\deg(\pi)}{\deg(\sigma)}\right)\cdot \left(\frac{d(\pi)}{d(\sigma)}\right)^{-1} =
\left(\frac{\mu_{1,\bG}}{\mu_{2,\bG}}\right)^{-1} \left(\frac{\mu_{1,\bM}}{\mu_{2,\bM}}\right).
\]
Note that this quotient is independent of the choice of $\psi$, so we may assume that the self-dual measure on $F$ associated to $\psi$ is the standard additive measure, namely $\vol(\fo)=1$. 

By definition, $\mu_{1,\bG}(K_G/K_G\cap A_{\bG})=1$. To deal with $\mu_{2,\bG}$, we fix a split form $\mathscr{G}$ of $\bG$ over $\fo$ and $\eta_{\bG}:\bG\xrightarrow{\sim} \mathscr{G}$ an $\overline{F}$-isomorphism. Let $\lieg$ be the Lie algebra of $\mathscr{G}$ with decomposition 
$$\lieg=\liez\oplus\lieg_{\der},$$
where $\liez$ (resp. $\lieg_{\der}$) is the center (resp. the derived subalgebra) of $\lieg$, both defined over $\fo$. Let $\omega_1$ (resp. $\omega_2$) be a top form on $\liez$ (resp. $\lieg_{\der}$) over $\fo$ with nonzero reduction. Then $\omega:=\omega_1\wedge \omega_2$ is a top form over $\fo$ on $\lieg$ with nonzero reduction. Let $\nu_{\bG}$ be the Haar measure on $G$ defined by $\omega$ and $\psi$.

An invariant top form on $\mathbb{G}_m=\Spec ~\fo [x^{\pm 1}]$ with nonzero reduction is given by $\frac{dx}{x}$, and
$$\int_{\fo^\times} \frac{dx}{|x|}=dx(\fo-\varpi \fo)=(1-q^{-1}).$$
Since $\bA_{\bG}\isom (\mathbb{G}_m)^{\dim \bA_{\bG}}$ over $F$, $\mu_{2,\bG}$ is then characterised by
$$\mu_{2,\bG}(K_G/K_G\cap A_{\bG})=\nu_{\bG}(K_G)\cdot (1-q^{-1})^{-\dim \bA_{\bG}}.$$
Similarly, we have $\mu_{1,\bM}(K_M/K_M\cap A_{\bM})=1$ and
$$\mu_{2,\bM}(K_M/K_M\cap A_{\bM})=\nu_{\bM}(K_M)\cdot (1-q^{-1})^{-\dim \bA_{\bM}},$$
where $\nu_{\bM}$ is the Haar measure defined in the same way as above for $M$. Thus
\[
\left(\frac{\mu_{1,\bG}}{\mu_{2,\bG}}\right)^{-1} \left(\frac{\mu_{1,\bM}}{\mu_{2,\bM}}\right)=\frac{\nu_{\bG}(K_G)}{\nu_{\bM}(K_M)} \cdot (1-q^{-1})^{\dim \bA_{\bM}-\dim \bA_{\bG}}.
\]

\begin{proposition}\label{measure quotient}Assume that $\bG$ is unramified (i.e. $\bG$ is quasi-split and splits over a finite unramified extension over $F$) and we take $K_G$ to be hyperspecial. Then
  \[
  \frac{\nu_{\bG}(K_G)}{\nu_{\bM}(K_M)}=\gamma(\mathbf{G}/\mathbf{M}),
  \] 
  and 
  \[
  \left(\frac{\deg(\pi)}{\deg(\sigma)}\right)\cdot \left(\frac{d(\pi)}{d(\sigma)}\right)^{-1}=\gamma(\mathbf{G}/\mathbf{M})\cdot (1-q^{-1})^{\dim \bA_{\bM}-\dim \bA_{\bG}}.
  \]
\end{proposition}

\hspace{\fill}\\

First we recall some basic facts about the motives associated to reductive groups introduced in \cite{gross-motive}. 

Let $\bH$ be a quasi-split connected reductive group over $F$, with a fixed maximal $F$-torus $\bT$. Let $W$ be the absolute Weyl group of $(\bH,\bT)$ and $E:=X^*(\bT)\otimes \QQ$. Since the Galois group $\Gamma=\Gal(\overline{F}/F)$ acts on $W$ and the (absolute) root datum of $\bH$, $E$ becomes a representation of $W\rtimes \Gamma$ over $\QQ$.

Let $R:=\Sym^*(E)^W$ and $R_+$ the ideal of elements of degree $\geq 1$ in $R$. Define the graded vector space
\[
V=R_+/R_+^2=\bigoplus _{d\geq 1} V_d
\]
with the induced grading, where $V_d$ is the summand of primitive invariants of degree $d$. Then each $V_d$ is a representation of $\Gamma$, and by results of Chevalley and Steinberg, $V$ is isomorphic to $E$ as representations of $\Gamma$ over $\QQ$, so that $\dim V=\dim \bT$. The motive $\mathfrak{M}_{\bH}$ of $\bH$ is defined to be
\[\mathfrak{M}_{\bH}=\bigoplus_{d\geq 1} V_d(1-d)
\]
where $V_d(k)$ is the $k$-Tate twist of $V_d$, and one has
\[
\dim \bH=\sum_{d\geq 1}(2d-1)\dim V_d.
\]

\begin{proof}[Proof of Proposition 3.1] 
Since $\bG$ is unramified, there exists a connected reductive group scheme $\mathscr{G}$ over $\fo$ such that $\bG$ is the generic fibre of $\mathscr{G}$ and $K_G:=\mathscr{G}(\fo)$ is a hyperspecial maximal compact subgroup of $G=\bG(F)$. In this case, it is also known that indeed $\mathscr{G}$ is also quasi-split over $\fo$ (cf. \cite[\textsection5.2.12]{conrad-reductive-group-schemes}).

So we may fix a Borel subgroup $\bB$ of $\bG$ with its Levi decomposition $\bB=\bT\bU$, such that $\bB,\bT,\bU$ are defined over $\fo$. Let $\bT_0$ be the maximal $F$-split subtorus of $\bT$. Given $\alpha\in \Sigma_{\red}(\bG,\bT_0)$, we denote by $\bU_\alpha$ its (relative) root subgroup. This is a closed subgroup normalized by $\bT_0$, and $\Lie(\bU_{\alpha})$ is the direct sum of the root spaces of $\{\alpha,2\alpha\}\cap \Sigma(\bG,\bT_0)$.  Then the subgroup $I_G$ of $G$ generated by
\[
 \bT(\fo),~\bU_{\alpha}(\fo)~(\alpha\in\Sigma_{\red}^+(\bG,\bT_0)),~\bU_{\alpha}(\fp)~(\alpha\in\Sigma_{\red}^-(\bG,\bT_0))
\]
is an Iwahori subgroup of $G$ with pro-unipotent radical $I_G^+$ generated by
\[
 \bT(1+\fp),~\bU_{\alpha}(\fo)~(\alpha\in\Sigma_{\red}^+(\bG,\bT_0)),~\bU_{\alpha}(\fp)~(\alpha\in\Sigma_{\red}^-(\bG,\bT_0)).
\]
We define $I_M,I_M^+ $ similarly for $\bM$, and $K_M:=K_G\cap M$.

Now let $\mathfrak{M}_{\bG}=\bigoplus_{d\geq 1}V_d(1-d)$ be the motive of $\bG$. We have
\[
\nu_{\bG}(I_G^+)=q^{-\mathfrak{v}_{\bG}},\quad\mathfrak{v}_{\bG}:=\frac{1}{2} a(\mathfrak{M}_{\bG}) +\sum_{d\geq 1} (d-1)\dim V_d^{I_F} + \rank_{F^{\ur}}\bG,
\]
where $a(\mathfrak{M}_{\bG})$ is the Artin conductor of $\mathfrak{M}_{\bG}$,  $I_F$ is the inertia group of $F$ and $F^{\ur}$ the maximal unramified extension of $F$ in $\overline{F}$ (cf. \cite[pp. 295]{gross-motive} and \cite[\S5]{gan-gross-haar-measure}). Since $\bG$ is unramified, 
\[
a(\mathfrak{M}_{\bG})=0,\quad V_d^{I_F}=V_d,\quad \rank_{F^{\ur}}\bG=\dim \bT.
\]
Since
\[
\dim \bG=\sum_{d\geq 1} (2d-1)\dim V_d,\quad\dim \bT=\sum_{d\geq 1} \dim V_d ,
\]
we get
\[
\nu_{\bG}(I_G^+)=q^{-\frac{1}{2}(\dim \bG+\dim \bT)}.
\]
Similarly, 
\[
\nu_{\bM}(I_M^+)=q^{-\frac{1}{2}(\dim \bM+\dim \bT)},
\]
so that
\begin{align*}
     \frac{\nu_{\bG}(K_G)}{\nu_{\bM}(K_M)}&=\frac{[K_G:I_G^+]}{[K_M:I_M^+]}\cdot \frac{\nu_{\bG}(I_{G}^+)}{\nu_{\bM}(I_{M}^+)}\\
     &=\frac{[K_G:I_G^+]}{[K_M:I_M^+]}\cdot q^{-\frac{1}{2}(\dim \bG-\dim \bM)}=\frac{[K_G:I_G^+]}{[K_M:I_M^+]}\cdot q^{-\dim \bN}.
\end{align*}

Next for any $\alpha\in \Sigma_{\red}(\bG,\bT_0)$, we have 

\begin{fact}
    $[\bU_\alpha(\fo):\bU_\alpha(\fp)]=q^{\dim \bU_\alpha}$.
\end{fact}

Let us just remark here that this follows directly from the structure of the relative root subgroups $\bU_{\alpha}$ and the unramified condition. We shall write down the related details as a proof later.

With this fact, for $\alpha\in \Sigma_{\red}^-(\bG,\bT_0)$ we have
\[
[U_\alpha\cap K_G:U_\alpha\cap I_G^+]=[\bU_\alpha(\fo):\bU_{\alpha}(\fp)]=q^{\dim \bU_\alpha}.
\]
Finally, for our choice of $K_G$, $I_G$, $I_G^+$,
\[\gamma(\mathbf{G}/\mathbf{M})=\frac{[K_G:I_G^+]}{[K_M:I_M^+]}\underset{\alpha\in \Sigma_{\red}(\Bar{\bP})}{\prod}[U_\alpha\cap K_G:U_\alpha\cap I_G^+]^{-1}\]
(cf. \cite[pp. 241]{waldspurger-plancherel}). Since $\Lie(\overline{\bN})=\bigoplus_{\alpha\in \Sigma_{\red}(\overline{\bP})}\Lie (\bU_{\alpha})$ and $\sum_{\alpha\in \Sigma_{\red}(\overline{\bP})}\dim \bU_{\alpha}=\dim \overline{\bN}=\dim \bN$, we get 
\[
\gamma(\mathbf{G}/\mathbf{M})=\frac{[K_G:I_G^+]}{[K_M:I_M^+]}\cdot q^{-\dim \bN}=\frac{\nu_{\bG}(K_G)}{\nu_{\bM}(K_M)}.
\]
\end{proof}

Now we recollect the related details of the above fact. The reader is referred to \cite[pp. 96-97]{kaletha} for an excellent and clear explanation of the structure of relative root subgroups of quasi-split groups. With this the fact is nothing more than a direct consequence of the unramified condition.

\begin{proof}[Proof of Fact 3.2]
   Let $\alpha\in \Sigma_{\red}(\bG,\bT_0)$. There are two cases. We take a preimage $\Tilde{\alpha}\in \Sigma(\bG,\bT)$ of $\alpha$ in the absolute root system. Recall that the Galois group $\Gamma=\Gal(\overline{F}/F)$ acts on $\Sigma(\bG,\bT)$. Let $\Gamma_{\Tilde{\alpha}}\subset \Gamma$ be the stabilizer of $\Tilde{\alpha}$ and $F_{\Tilde{\alpha}}\subset \overline{F}$ be the fixed field, so that $F_{\Tilde{\alpha}}/F$ is unramified by our assumption. We then denote by $\fo_{\Tilde{\alpha}}$ the ring of integers of $F_{\Tilde{\alpha}}$, with the maximal ideal $\fp_{\Tilde{\alpha}}$. 
\begin{itemize}
    \item  
   First, suppose $2\alpha\notin \Sigma(\bG,\bT_0)$.  In this case the natural $F_{\Tilde{\alpha}}$-embedding $\bU_{\Tilde{\alpha}}\hookrightarrow \bU_{\alpha}$ induces an $F$-isomorphism 
\[
\Res_{F_{\Tilde{\alpha}}/F}\bU_{\Tilde{\alpha}}\xrightarrow{\sim} \bU_{\alpha}
\]
(cf. \cite[pp. 97]{kaletha}). In our setting, $\bU_{\alpha}$ and the above isomorphism are defined over $\fo$. Since $F_{\Tilde{\alpha}}/F$ is unramified, 
   \begin{align*}
    [\bU_{\alpha}(\fo):\bU_{\alpha}(\fp)]
    &=[\bU_{\Tilde{\alpha}}(\fo_{\Tilde{\alpha}}):\bU_{\Tilde{\alpha}}(\fp_{\Tilde{\alpha}})]\\
    &=[\fo_{\Tilde{\alpha}}:\fp_{\Tilde{\alpha}}]=q^{[F_{\Tilde{\alpha}}:F]}=q^{\dim \bU_{\alpha}}.
    \end{align*}
\item 
    The second situation is $2\alpha\in \Sigma(\bG,\bT_0)$. In this case the structure of the relative root subgroup $\bU_\alpha$ is more complicated to explain. Indeed this case only occurs when $\Sigma(\bG,\bT)$ has an irreducible factor of type $A_{2n}$ and $\Gamma$ acts nontrivially on this type $A_{2n}$ factor.

    The preimage of $\alpha$ in $\Sigma(\bG,\bT)$ is of even order, and there is a unique $\Tilde{\alpha}'$ in the preimage which is not orthogonal to $\Tilde{\alpha}$. Then $\Tilde{\beta}:=\Tilde{\alpha}+\Tilde{\alpha}'$ lies in $\Sigma(\bG,\bT)$ and is mapped to $2\alpha\in \Sigma(\bG,\bT_0)$. We denote by $F_{\Tilde{\beta}},\fo_{\Tilde{\beta}},\fp_{\Tilde{\beta}}$ the similar notations for $\Tilde{\beta}$ (as those for $\Tilde{\alpha}$); then $F_{\Tilde{\alpha}}/F_{\Tilde{\beta}}$ is quadratic. We have $\bU_\alpha\simeq \Res_{F_{\Tilde{\beta}}/F}\bU_{[\Tilde{\alpha}]}$ over $\fo$, where $\bU_{[\Tilde{\alpha}]}$ is isomorphic to the unipotent radical of a Borel subgroup of the quasi-split $\SU_3(F_{\Tilde{\alpha}}/F_{\Tilde{\beta}})$ (\cite[pp. 97]{kaletha}). More explicitly, if we denote by $a\mapsto \overline{a}$ the nontrivial element of $\Gal(F_{\tilde{\alpha}}/F_{\tilde{\beta}})$, then 
    \[
    \bU_{[\Tilde{\alpha}]}=\{u(a,b):=\begin{bmatrix}
        1 & a & b \\ & 1 & \overline{a} \\ && 1
    \end{bmatrix}\mid a,b\in F_{\tilde{\alpha}},~a\overline{a}=b+\overline{b}\}
    \]
    as a group over $F_{\tilde{\beta}}$. We have an exact sequence 
    \[
    1\rightarrow \bU^0_{[\tilde{\alpha}]}\rightarrow \bU_{[\tilde{\alpha}]}\rightarrow \Res_{F_{\tilde{\alpha}}/F_{\tilde{\beta}}}\bG_a\rightarrow 1
    \]
  of $F_{\tilde{\beta}}$-groups, where $\bG_a$ is the additive group of $F_{\tilde{\alpha}}$ and 
    \[
    \bU^0_{[\tilde{\alpha}]}=\{u(0,b)\mid b\in F_{\tilde{\alpha}},~b+\overline{b}=0\}.
    \]
    
  Since $F_{\tilde{\alpha}}\supset F_{\tilde{\beta}}\supset F$ are unramified, $\bU^0_{[\tilde{\alpha}]}(\fo_{\tilde{\beta}})/\bU^0_{[\tilde{\alpha}]}(\fp_{\tilde{\beta}})$ is just the subgroup of $\fo_{\alpha}/\fp_{\alpha}$ consisting of elements of quadratic trace $0$, hence of order $[\fo_{\tilde{\beta}}:\fp_{\tilde{\beta}}]$, and
\begin{align*}
    [\bU_{\alpha}(\fo):\bU_\alpha(\fp)]&=[\bU_{[\Tilde{\alpha}]}(\fo_{\Tilde{\beta}}):\bU_{[\Tilde{\alpha}]}(\fp_{\Tilde{\beta}})]=[\fo_{\tilde{\beta}}:\fp_{\tilde{\beta}}]^3\\
    &=q^{[F_{\tilde{\beta}}:F]\cdot \dim_{F_{\tilde{\beta}}}(U_{[\tilde{\alpha}]})} =q^{\dim \bU_\alpha}.
\end{align*}
    \end{itemize}  
\end{proof}

\begin{remark}
    If $\bG$ is an arbitrary connected reductive group over $F$ which merely satisfies one of the following conditions:
    \begin{enumerate}
        \item quasi-split;
        \item splits over a finite unramified extension over $F$,
    \end{enumerate}
    then the above proposition does not hold in general. (For some examples, cf. Appendix E of \cite{gan-ichino-formal-degree}, the tables of case A.) For this reason, we add the unramified condition for our main theorem to make the explicit formula less technical.
\end{remark}

\section{Tools from Harmonic Analysis}\label{section harmonic analysis}
In the rest of this paper, except for \S\ref{subsection construction of discrete parameters}, $\bG$ will denote a quasi-split connected reductive group over $F$ with a fixed Borel subgroup $\bB=\bT\bU$, and $\bT_0$ the maximal $F$-split subtorus of $\bT$. We then use standard notations explained at the end of \S\ref{subsection general notations}.

In this section, we recall Shahidi's results and apply Heiermann's formula to our case. The main difficulty of the latter is an explicit comparison between two different measures on the orbit. 

Throughout \S 4, we fix
\begin{itemize}
    \item a standard maximal parabolic subgroup $\bP=\mathbf{M}\mathbf{N}$ of our quasi-split group $\mathbf{G}$ corresponding to $\Delta-\{\alpha\}$ for a positive simple root $\alpha\in \Delta$;
    \item a special maximal compact subgroup $K_G\subset G=\bG(F)$. 
\end{itemize}
For any closed subgroup $H$ of $G$, we equip $H$ with the Haar measure such that $K_G\cap H:=K_H$ has volume one.

\subsection{Results from Langlands-Shahidi method}\label{subsection shahidi}
 
Let $\rho_{\bP}\in\fa_{\mathbf{M}}^*$ be the half sum of positive roots appearing in $\mathbf{N}$ and
$$\Tilde{\alpha}:=\langle \rho_{\bP},\alpha^\vee\rangle^{-1} \rho_{\bP}\in \fa_{\mathbf{M}}^*$$
be the corresponding fundamental weight. Denote by $r$ the adjoint representation of ${^L\!M}$ on $\widehat{\mathfrak{n}}:=\Lie(\widehat{N})$. Then $r$ decomposes as $(r,\widehat{\mathfrak{n}})=\bigoplus_{i=1}^m (r_i,V_i)$, where each irreducible component $(r_i,V_i)$ is given by
$$V_i:=\{X_{\beta^\vee}\mid \beta\in \Sigma, \langle \Tilde{\alpha},\beta^\vee\rangle=i\},\quad 1\leq i\leq m,$$
cf. \cite[Proposition 4.1]{shahidi-ramanujan}. Let $W=W^{\bG}(\bT_0)$ be the Weyl group of $\bG$ with respect to the maximal split torus $\bT_0$ and $w_0\in W=W^{\bG}(\bT_0)$ be the element of the Weyl group characterized by $w_0(\Delta-\{\alpha\})\subset \Delta$ and $w_0\alpha\in \Sigma^-$, i.e. the longest element modulo the Weyl group of $\bM$.

 Generic discrete series of $G$ supported on a maximal Levi subgroup were classified by Shahidi as follows.
\begin{theorem}[{\cite[ Theorem 8.1]{shahidi-plancherel}}]\label{shahidi classification}
    Let $\pi$ be a generic discrete series of $G$ supported on $\mathbf{M}$. Then there exists an irreducible generic unitary supercuspidal representation $\sigma$ of $M$ satisfying the following properties:
    \begin{enumerate}
        \item $\sigma$ is ``ramified", i.e. $w_0\sigma\isom \sigma$, and $i_P^G\sigma$ is irreducible;
        \item there exists a unique $j\in\{1,2\}$, such that $L(s,\sigma,r_j)$ has a simple pole at $s=0$;
        \item the representation $i_P^G\sigma_{\Tilde{\alpha}/j}$ (which is of length two) has a unique generic irreducible subrepresentation isomorphic to $\pi$, and its Langlands quotient is non-generic, non-tempered, and pre-unitary.
    \end{enumerate}
\end{theorem}
Since by \cite[Proposition 7.8]{shahidi-plancherel} we have (for any irreducible generic supercuspidal $\sigma$)
\begin{align*}
    \overline{\gamma(s,\sigma,r_i,\psi)}&=\gamma(\bar{s},\Tilde{\sigma},r_i,\Bar{\psi})\\
    &=\varepsilon(\bar{s},\Tilde{\sigma},r_i,\Bar{\psi})\cdot \frac{L(1-\Bar{s},\sigma,r_i)}{L(\bar{s},\Tilde{\sigma},r_i)},
\end{align*}
``$L(s,\sigma,r_i)$ has a simple pole at $s=0$" could be paraphrased as ``$\gamma(s,\sigma,r_i,\psi)$ has a simple pole at $s=1$".

Next we recall the definition of Harish-Chandra $\mu$-function. The intertwining operator $J_{\bar{P}|P}(\sigma):i_P^G\sigma\rightarrow i_{\bar{P}}^G\sigma$ is defined to be
$$J_{\bar{P}|P}(\sigma)f(g):=\int_{\bar{N}}f(\bar{n}g)d\bar{n},\quad (f\in i_P^G\sigma, g\in G),$$
and the Harish-Chandra $\mu$-function is a rational function $\mu$ on $\cO=\cO_{\sigma}$ defined by
$$J_{P|\bar{P}}(\sigma')\circ J_{\bar{P}|P}(\sigma')=\mu(\sigma')^{-1}$$
for $\sigma'$ in a Zariski open subset of $\cO$. Note that the definition here is compatible with \cite[\S 1.5]{heiermann-spectrales} but differs from that in \cite{shahidi-plancherel}\cite{waldspurger-plancherel}. When $\bP=\bM\bN$ is maximal, the $\mu$-function here is equal to $\gamma(\bG/\bM)^{-2}\mu$ defined in \cite{shahidi-plancherel}\cite{waldspurger-plancherel}.

The main theorem 3.5 and corollary 3.6 of \cite{shahidi-plancherel} then reads
\begin{theorem}\label{shahidi main theorem}
    Let $\sigma$ be an irreducible generic unitary supercuspidal representation of $M$. Then
    \[\mu(\sigma\otimes \chi_{s\Tilde{\alpha}})=\prod_{i=1}^m \gamma^{\Sh}(is,\sigma,r_i,\Bar{\psi})\gamma^{\Sh}(-is,\Tilde{\sigma},r_i,\psi),\]
    where $\gamma^{\Sh}$ are Shahidi's $\gamma$-factors characterized in \cite[Theorem 3.5]{shahidi-plancherel}.
\end{theorem}
\begin{remark} Assuming the L-parameter $\varphi_{\sigma}$ of $\sigma$ (conjectural in general), it is believed that these $\gamma^{\Sh}(s,\sigma,r_i,\psi)$ should coincide with corresponding Artin $\gamma$-factors $\gamma(s,r_i\circ\varphi_{\sigma},\psi)$ obtained from the L-parameter. Though verified in many cases, this has not yet been established in full generality.
\end{remark}

\subsection{Heiermann's formula}\label{subsection Heiermann's formula}
In this subsection, we then fix 
\begin{itemize}
    \item  a generic discrete series representation $\pi$ of $G$ supported on $\mathbf{M}$; 
    \item an irreducible unitary supercuspidal $\sigma$ of $M$, and a unique $j\in\{1,2\}$, such that $\pi$ is the unique irreducible generic subrepresentation of $i_P^G(\sigma_{\Tilde{\alpha}/j})$. 
\end{itemize} 

In \cite[\textsection 8.6]{heiermann-spectrales}, as a byproduct of his spectral decomposition, Heiermann established a formula of $\deg(\pi)/\deg(\sigma)$ in terms of the residues of the $\mu$-function. To apply his formula to our case, we need some preparations.

First let $W=W^{\bG}(\bT_0)$ be the Weyl group of $(\mathbf{G},\mathbf{T}_0)$. Following notations of \cite{heiermann-spectrales}, let 
$$\mathcal{W}(\mathbf{M}):=\{w\in W\mid w\mathbf{M} w^{-1}=\mathbf{M}\}$$
and $W^{\mathbf{M}}$ be the Weyl group of $\mathbf{M}$, identified as a normal subgroup of $\mathcal{W}(\mathbf{M})$. Let $W(\mathbf{M}):=\mathcal{W}(\mathbf{M})/W^{\mathbf{M}}\isom N_G(\bM)/M$.

\begin{lemma}
    $|W(\mathbf{M})|\leq 2$.
\end{lemma}
\begin{proof}
    Let $\ZZ\langle \theta\rangle$ be the free abelian subgroup of $X^*(\mathbf{T}_0)$ generated by $\theta:=\Delta-\{\alpha\}$ and $\Sigma_{\theta}:=\ZZ\langle \theta\rangle \cap \Sigma$ the root subsystem of $\Sigma$ generated by $\theta$, thus $\Sigma_{\theta}=\Sigma(\mathbf{M},\mathbf{T}_0)$. Then $\mathcal{W}(\mathbf{M})$ is the subgroup of $W$ preserving $\Sigma_{\theta}$, or equivalently, preserving $\ZZ\langle \theta\rangle$.

    For any other system $\theta'$ of simple roots of $\Sigma_{\theta}$, there exists $w'\in W^{\mathbf{M}}$ such that $w'\theta'=\theta$. Therefore any $\overline{w}\in W(\mathbf{M})$ admits a representative $w$ with $w\theta=\theta$. Since now $w\Delta=w\theta\cup \{w\alpha\}=\theta\cup \{w\alpha\}$ spans $\Sigma$ over $\ZZ$, $w\alpha\in (\pm \alpha +\ZZ\langle \theta\rangle)$.

    If $w\alpha\in ( \alpha +\ZZ\langle \theta\rangle)$, then since $w\Delta=\theta\cup \{w\alpha\}$ and $\Delta$ are two systems of simple roots, $w\alpha\in ( \alpha +\ZZ_{\geq 0}\langle \theta\rangle)$. In this case $w\Delta\subset \Sigma^+$, hence $w\Sigma^+\subset \Sigma^+$ and $w=1$.

    Thus if $\Bar{w}\in W(\mathbf{M})$ is a nontrivial element, it has a representative $w$ with $w\theta=\theta$ and $w\alpha\in (-\alpha+\ZZ\langle \theta\rangle)$. But if $w_1,w_2$ are two such elements of $W$, then $w_1^{-1}w_2\theta=\theta$ and $w_1^{-1}w_2\alpha\in (\alpha+\ZZ\langle \theta\rangle)$, which implies $w_1=w_2$ by the former case. This completes the proof.
\end{proof}

Recall that $\fa_{\mathbf{M}}^*=\fa_{\mathbf{G}}^*\oplus \fa_{\mathbf{M}}^{\mathbf{G} *}$ and $\rho_{\bP}\in \fa_{\mathbf{M}}^{\mathbf{G} *}\isom \RR$ . The corresponding fundamental weight 
    \[
    \Tilde{\alpha}:=\langle \rho_{\bP},\alpha^\vee\rangle ^{-1}\rho_{\bP}
    \]
    is then characterized as an element of $\fa_{\mathbf{M}}^{\mathbf{G} *}$ satisfying
    \[
    \langle \Tilde{\alpha},\alpha^\vee\rangle=1.
    \]
Since we only care about the ``relative position" of the twisted characters in inductions, let 
\begin{itemize}
\item $X^*(\bA_{\bM})^{\bG}=X^*(\bA_{\bM}/\bA_{\bG})\subset X^*(\bA_{\bM})$ and $X^*(\bM)^{\bG}\subset X^*(\bM)$ be the subgroup of characters of $\bM$ trivial on $\bA_{\bG}$;
    \item $\chi$ be the generator of $X^*(\mathbf{M})^{\mathbf{G}}\isom \ZZ$ with $\langle \chi,\alpha^\vee\rangle >0$ (here we identify $\chi\in X^*(\bM)$ with an element of $X^*(\bT_0)$ through $\fa_0^*=\fa_{\bM}^*\oplus \fa_0^{\bM*}$, cf. \S\ref{subsection preliminaries harmonic});
    \item  $X^{\ur}(M)^G$ (resp. $X_0^{\ur}(M)^G$) be the image of $\mathfrak{a}_{\bM,\CC}^{\bG *}$ (resp. $i\mathfrak{a}_{\bM}^{\bG *}$) under the map $\mathfrak{a}_{\bM,\CC}^{\bG *}\twoheadrightarrow X^{\ur}(M)$ (note that the group $X^{\ur}(M)^G$ here is denoted by $X_M^G$ in \cite[pp. 7]{eisenstein});
    \item $\mathcal{O}^G$ (resp. $\mathcal{O}_0^G$) be the $X^{\mathrm{ur}}(M)^G$-orbit (resp. $X^{\mathrm{ur}}_0(M)^G$-orbit) of $\sigma$, namely
    \[
    \mathcal{O}^G=\{\sigma\otimes \chi\mid \chi\in X^{\mathrm{ur}}(M)^G\},\quad \mathcal{O}_0^G=\{\sigma\otimes \chi\mid \chi\in X_0^{\mathrm{ur}}(M)^G\}
 \]
     in the sense of isomorphism classes.
\end{itemize}
For our $\pi,\sigma$, let $A:=\{\sigma_{\Tilde{\alpha}/j}\}\subset \mathcal{O}^G$ be the corresponding ``root hyperplane" and
$$\mathrm{Stab} (A):=\{w\in W(\mathbf{M})\mid wA=A\}=\{w\in W(\mathbf{M})\mid w\sigma_{\Tilde{\alpha}/j}\isom \sigma_{\Tilde{\alpha}/j}\}$$
as introduced in \cite[\textsection 6.1]{heiermann-spectrales}. Since $|W(\mathbf{M})|\leq 2$, and by Shahidi's classification $w_0\sigma\isom\sigma$, we get the following
\begin{corollary}
    $|\mathrm{Stab}(A)|=1$.
\end{corollary}
\begin{proof}
     By Shahidi's classification, $w_0\sigma\isom \sigma$, where $w_0$ is determined by $w_0(\Delta-\{\alpha\})\subset \Delta$ and $w_0\alpha\in \Sigma^-$.
     So by the above lemma $W(\bM)=\{1,\overline{w_0}\}$, but $w_0$ maps $\Tilde{\alpha}$ to a negative Weyl chamber, thus $w_0\sigma_{\Tilde{\alpha}/j}\not\simeq \sigma_{\Tilde{\alpha}/j}$ and $\mathrm{Stab} (A)=\{1\}$.
\end{proof}

Next we let $y_1\in\RR^\times_{>0}$ be the smallest positive real number such that $\sigma\otimes \chi_{i y_1\Tilde{\alpha}}\isom \sigma$, so that $y_1\in \QQ^\times _{>0}\cdot \left(\frac{2\pi}{\log q}\right)$ and
\[
\cO_0^G=\{\sigma\otimes \chi_{i y_1\Tilde{\alpha}}\mid y\in [0,y_1)\}.
\]
We fix the basis $\Tilde{\alpha}$ of $\fa_{\mathbf{M},\CC}^{\mathbf{G}*}=\CC\cdot \Tilde{\alpha}$. 

In our case, the induced representation $i_P^G\sigma_{\Tilde{\alpha}/s}$ is irreducible for $s>0$ and $s\neq 1/j$ by Theorem 8.1 of \cite{shahidi-plancherel}. Hence on the half line $s\Tilde{\alpha}$, $s>0$, the unique pole of $\mu(\sigma\otimes\chi_{s\Tilde{\alpha}})$ is $s=1/j$ (\cite[Lemma 5.4.2.4]{silberger}). Now applying the residue theorem in complex analysis to the contour
\[
 \begin{tikzpicture}
  \draw (0.4,1.2)node[above]{$iy_1$};
  \draw (-0.4,0)node[below]{$O$};
  \draw (4,0)node[below]{$r\gg 0$};
  \draw (4,1.2)node[above]{$r+iy_1$};
  \draw (1.2,-0.04)node[below]{$1/j$};
 \draw (-1,0)--(0,0);
 \draw (0,1.8)--(0,-1);
  \draw (4,0)--(5,0);
\draw[thick] (0,0) -- (1.1,0) arc(180:360:0.1) -- (4,0) ;
\draw[thick] (0,1.2) -- (1.1,1.2) arc(180:360:0.1) -- (4,1.2) ;
         \draw[thick] (0,0)--(0,1.2);
           \draw[thick] (4,1.2)--(4,0);
           \node at (1.2,0)[circle,fill,inner sep=1pt]{};
           \node at (1.2,1.2)[circle,fill,inner sep=1pt]{};
\node at (0,1.2)[circle,fill,inner sep=1pt]{};
\node at (4,1.2)[circle,fill,inner sep=1pt]{};
\node at (0,0)[circle,fill,inner sep=1pt]{};
\node at (4,0)[circle,fill,inner sep=1pt]{};
       \end{tikzpicture}
\]
we get 
\[
\int_0^{y_1}\mu(\sigma\otimes \chi_{(r+iy)\Tilde{\alpha}})dy=\int_0^{y_1}\mu(\sigma\otimes\chi_{iy\Tilde{\alpha}})dy+2\pi \cdot \underset{s=1/j}{\Res}\mu(\sigma\otimes \chi_{s\Tilde{\alpha}}).\quad (r\gg 0)
\]

However, the measure on $\cO_0^G$ determined by the coordinate $\Tilde{\alpha}$ as above is not the standard one used in local harmonic analysis. The standard one is defined as follows. One equips $X_0^{\ur}(A_{\bM})^G$ with the Haar measure such that it has total volume one, and then requires that
\[
\res:X_0^{\ur}(M)^G\rightarrow X_0^{\ur}(A_{\bM})^G
\]
and 
\[
X_0^{\ur}(M)^G\twoheadrightarrow \cO_0^G
\]
preserve the measures locally. Then with respect to this standard measure on $\cO_0^G$, we have
\begin{equation}\label{residue operator}
    \int_{\cO_0^G} \mu(\sigma\otimes \chi_{r\Tilde{\alpha}})d\sigma=\int_{\cO_0^G} \mu(\sigma)d\sigma+\frac{\vol(\cO_0^G)}{y_1}\cdot 2\pi \cdot \underset{s=1/j}{\Res}\mu(\sigma\otimes \chi_{s\Tilde{\alpha}}).\quad (r\gg 0)
\end{equation}

In order to make explicit use of Heiermann's formula, we shall next compute the quotient $\vol(\cO_0^G)/y_1$ explicitly.

Suppose
    \[
    |\chi(M)|=q^{l\ZZ}
    \]
    for some $l\in \ZZ_{\geq 1}$, then we have
    \[
    X^{\ur}_0(M)^{G}=\{|\chi|^{\left(\frac{2\pi i}{\log q}\right)s}\mid s\in [0,\frac{1}{l})\}.
    \]
Let 
    \[\mathrm{Stab}_M^G(\sigma):=\{\chi\in X^{\ur}(M)^G\mid \sigma\otimes \chi\isom \sigma\}.
    \]
    This is a finite subgroup of $X_0^{\ur}(M)^{G}$. We denote its order by $t$ (the ``torsion number"). 
    Then the smallest positive number $s_1\in \RR_{> 0}^\times$ such that 
    \[
    \sigma\otimes |\chi|^{is_1}\isom \sigma
    \]
    is clearly
    \[
    s_1=\left(\frac{2\pi }{\log q}\right)\cdot \frac{1}{lt}.
    \]
     Since $\Tilde{\alpha}=\langle \chi,\alpha^\vee\rangle^{-1}\chi$, the smallest $y_1\in \RR^\times_{>0}$ such that
    \[
    \sigma\otimes \chi_{iy_1\Tilde{\alpha}}\isom \sigma
    \]
    (recall that $\chi_{iy_1\Tilde{\alpha}}$ is the image of $iy_1\Tilde{\alpha}$ under $\fa_{\bM,\CC}^{\bG*}\twoheadrightarrow X_0^{\ur}(M)^G$) is then
    \[
    y_1=\langle \chi,\alpha^\vee\rangle\left(\frac{2\pi }{\log q}\right)\cdot \frac{1}{lt}.
    \]

    To evaluate $\vol(\cO_0^G)$, let
    \[
    \res:X^*(\mathbf{M})^{\mathbf{G}}\rightarrow X^*(\mathbf{A}_{\mathbf{M}})^{\mathbf{G}}
    \]
    be the restriction morphism. Pick the basis element $\eta\in X^*(\mathbf{A}_{\mathbf{M}})^{\mathbf{G}}\isom \ZZ$ such that 
    \[
    \res(\chi)=\eta^m
    \]
    for some $m\in \ZZ_{\geq 1}$. Now 
    \[
    X^{\ur}_0(A_{\mathbf{M}})^G=\{|\eta|^{\left(\frac{2\pi i}{\log q}\right)s}\mid  s\in [0,1)\}
    \]
    and we already have
     \[
    X^{\ur}_0(M)^G=\{|\chi|^{\left(\frac{2\pi i}{\log q}\right)s}\mid s\in [0,\frac{1}{l})\}.
    \]
    The restriction morphism
    \[
    \res: X_0^{\ur}(M)^G\rightarrow X_0^{\ur}(A_{\mathbf{M}})^G,
    \]
    under these coordinates, maps $s$ to $ms$. This map preserves the measures locally, and $X_0^{\ur}(A_{\mathbf{M}})^G$ has volume $1$. Thus we have
    \[
    \vol(X_0^{\ur}(M)^G)=\frac{m}{l},~\vol(\cO_0^G)=\frac{m}{lt}.
    \]

    Finally we get
    \[
    \frac{\vol(\cO_0^G)}{y_1}=\left(\frac{\log q}{2\pi}\right)\cdot \frac{m}{\langle \chi,\alpha^\vee\rangle}.
    \]

    Recall that we've fixed a special maximal compact subgroup $K_G$ of $G$ at the beginning of this section, and for any closed subgroup $H$ of $G$, $H$ is equipped with the measure such that $K_H:=H\cap K_G$ has volume one. We denote by $\deg(\cdot)$ the formal degree with respect to this choice of the measure. Then by Heiermann's formula in \cite[\textsection 8.6]{heiermann-spectrales}, we obtain the following 
\begin{theorem}\label{Heiermann's formula}
\begin{align*}
    \frac{\deg(\pi)}{\deg(\sigma)}&=\gamma(\mathbf{G}/\mathbf{M})\cdot |\mathrm{Stab}(A)|^{-1}\cdot \frac{\vol(\cO_0^G)}{y_1}\cdot 2\pi \cdot \underset{s=1/j}{\Res}\mu(\sigma\otimes \chi_{s\Tilde{\alpha}})\\
    &=\gamma(\mathbf{G}/\mathbf{M})\cdot \log q\cdot \frac{m}{\langle \chi,\alpha^\vee\rangle } \cdot \underset{s=1/j}{\Res}\mu(\sigma\otimes \chi_{s\Tilde{\alpha}})
\end{align*}
 where 
    \begin{itemize}
        \item $\chi\in X^*(\mathbf{M})^{\mathbf{G}}\isom \ZZ$ is the generator with $\langle \chi,\alpha^\vee\rangle>0$;
        \item $m$ is the index of $\res(X^*(\mathbf{M})^{\mathbf{G}})$ in $X^*(\mathbf{A}_{\mathbf{M}})^{\mathbf{G}}$.
    \end{itemize}
    
\end{theorem}

\begin{remark}
    The definition of 
    \begin{itemize}
        \item the formal degree quotient on the left hand side,
        \item the constant $\gamma(\bG/\bM)$,
        \item and the $\mu$-function
    \end{itemize}   
    all depend on the choice of $K_G$.
\end{remark}

\begin{proof}
   The general formula of \cite[\S 8.6]{heiermann-spectrales} reads
   \[
    \frac{\deg(\pi)}{\deg(\sigma)}=\gamma(\mathbf{G}/\mathbf{M})\cdot |\mathrm{Stab}(A)|^{-1}\cdot m(\pi,\sigma)\cdot (\Res_A^P\mu)(\sigma).
   \]
   
   The multiplicity $m(\pi,\sigma)$ above in our case is one, since $\pi$ is the unique generic subquotient of $i_P^G\sigma_{\Tilde{\alpha}/j}$ (cf. (3) of Theorem \ref{shahidi classification}). The only point that requires an explanation is the notation $(\Res_A^P\mu)(\sigma)$ of Heiermann. 
   
   This is a multiple residue operator defined in a more general setting for arbitrary $\bP$, not necessarily maximal. The general definition is quite tough, but in our case, since $\bP$ is maximal (i.e. of corank one), up to constant this operator is exactly $\Res_{s=1/j}\mu(\sigma\otimes\chi_{s\Tilde{\alpha}})$, and by \cite[\S4.11]{heiermann-spectrales} (we take $\psi$ there to be $\mu$, and as we mentioned before, $\Tilde{\alpha}/j\in \CC\cdot \Tilde{\alpha}$ is the unique pole of $\mu$ on the right half plane) it satisfies
   \[
    \int_{\cO_0^G} \mu(\sigma\otimes \chi_{r\Tilde{\alpha}})d\sigma=\int_{\cO_0^G} \mu(\sigma)d\sigma+(\Res_A^P\mu)(\sigma).\quad (r\gg 0)
   \]
Thus by our computation \eqref{residue operator} above, it is exactly 
\[
(\Res_A^P\mu)(\sigma)=\frac{\vol(\cO_0^G)}{y_1}\cdot 2\pi \cdot \underset{s=1/j}{\Res}\mu(\sigma\otimes \chi_{s\Tilde{\alpha}}),
\]
hence the result.
\end{proof}

We remark that the notation ``$\Tilde{\alpha}$" introduced in \cite[\S3.2]{heiermann-spectrales} is not the fundamental weight (following the convention of Shahidi) corresponding to $\alpha$ in our case. Since it is used to define the multiple residue operator in general, one has to be careful while applying the formula to more general cases.

\section{Parameters}\label{section parameters}
In this section, we summarize Heiermann's construction of discrete parameters and compute the adjoint $\gamma$-factors for discrete parameters assuming this construction.
\subsection{Conditional Construction of Discrete Parameters}\label{subsection construction of discrete parameters}
Only in this subsection, we let $\bG$ be a general connected reductive group over $F$ with a fixed maximal split torus $\bT_0$. We then fix a semi-standard parabolic subgroup $\mathbf{P}=\mathbf{M}\mathbf{N}$ of $\mathbf{G}$, i.e. $\bM\supset \bT_0$. Let $\pi$ (resp. $\sigma$) be an irreducible smooth representation of $G$ (resp. $M$). Suppose $\pi$ is a subquotient of $i_P^G\sigma_\lambda$ for some $\lambda\in\fa_{\mathbf{M}}^*$. 

Assuming the LLC for $\mathbf{G}$ (resp. $\mathbf{M}$), which is conjectural in general, let $\varphi_\pi:\WD_F\rightarrow {^LG}$ (resp. $\varphi_\sigma:\WD_F\rightarrow {^L\!M}$) be the L-parameter of $\pi$ (resp. $\sigma$). Let ${^L\!P}$ be a (relevant) parabolic subgroup of ${^LG}$ corresponding to $\bP$ with inclusion $\iota_P:{^L\!P}\hookrightarrow {^LG}$; these are unique up to $\widehat{G}$-conjugacy. The conjectural LLC includes the following properties:
\begin{itemize}
    \item Suppose $\pi$ is non-tempered. According to the Langlands classification, $\pi$ is the Langlands quotient of $i_P^G \sigma_{\lambda}$ for some tempered $\sigma$. Then
    \begin{equation}\label{non-tempered parameters}
        \varphi_{\pi}=\iota_P\circ \varphi_{\sigma_\lambda}.
    \end{equation}
    \item Suppose $\pi$ is tempered. According to the classification of tempered representations, $\pi$ is a subrepresentation of $i_P^G\sigma$ for some discrete series $\sigma$. Then
    \begin{equation}\label{tempered parameters}
        \varphi_{\pi}=\iota_P\circ \varphi_{\sigma}.
    \end{equation}
\end{itemize}

When $\pi$ is a discrete series and $\sigma$ is unitary supercuspidal, it is a natural question to ask how to construct $\varphi_\pi$ from $\varphi_\sigma$. This problem is much subtler than the above two standard cases, which was studied by Heiermann in \cite{heiermann-orbites}\cite{heiermann-unipotent} using tools from local harmonic analysis developed by himeself. We now summarize his results as follows. 

Thus in the rest of this section, we shall assume the LLC for irreducible supercuspidal representations of all semi-standard Levi subgroups of $\mathbf{G}$. Let $\sigma$ be unitary supercuspidal, with L-parameter $\varphi_\sigma:\WD_F\rightarrow {^L\!M}$, and assume that $\pi$ is a discrete series subquotient of $i_P^G\sigma_\lambda$ for some $\lambda\in \fa_{\mathbf{M}}^*$. 

Let $T_{^L\!M}$ be the maximal torus of  $Z(\widehat{M})^\Gamma$. Let $s_\lambda\in T_{^L\!M}$ be the element corresponding to $(-\lambda)\in\fa_{\mathbf{M}}^*$ by the LLC for tori (cf. \S \ref{subsection parameters}, especially the final remark), and 
$$s_\sigma:=\varphi_\sigma(1,\begin{bmatrix}
    q^{\frac{1}{2}}&\\&q^{-\frac{1}{2}}
\end{bmatrix}).$$
Define
\begin{equation}\label{s-sigma-lambda}
    s_{\sigma,\lambda}:=s_\sigma s_\lambda\in {^L\!M}.
\end{equation}
Let $\widehat{M}^\sigma:=Z_{\widehat{G}}(\varphi_\sigma(\W_F))$; this is a complex reductive Lie group (cf. \cite[Proposition 5.2]{heiermann-orbites}) but not necessarily connected in general. We denote by $\widehat{M}^\sigma_{\der}$ its derived subgroup. 

Under some natural assumptions (cf. \cite[\textsection 5.2, 5.3]{heiermann-orbites}) on irreducible supercuspidal representations of Levi subgroups, Heiermann proved that there exists a nilpotent element $N_{\sigma,\lambda}\in \Lie(\widehat{M}^\sigma_{\der})$ (not unique in general), such that 
\begin{enumerate}
    \item $\Ad(s_{\sigma,\lambda})N_{\sigma,\lambda}=qN_{\sigma,\lambda}$;
    \item a torus in ${^LG}$ that centralizes $s_{\sigma,\lambda}$ and $N_{\sigma,\lambda}$ simultaneouly is contained in $Z(\widehat{M}^\sigma_{\der})$,
\end{enumerate}
i.e. $(s_{\sigma,\lambda},N_{\sigma,\lambda})$ is an $L^2$-pair (for more details of $L^2$-pair, cf. \cite[\S 3.1]{heiermann-orbites}\cite{lusztig}). It was proved that such an $L^2$-pair determines an algebraic homomorphism $\psi_{\sigma,\lambda}:\SL_2(\CC)\rightarrow \widehat{M}^\sigma$, with
$$\psi_{\sigma,\lambda}(\begin{bmatrix}
1 & 1 \\
0 & 1
\end{bmatrix})=\exp (N_{\sigma,\lambda}), ~\psi_{\sigma, \lambda}(\begin{bmatrix}
q^{1 / 2} & 0 \\
0 & q^{-1 / 2}
\end{bmatrix})=s_{\sigma, \lambda}.
$$

Based on these results, Heiermann associated to $\pi$ an L-parameter $\varphi_\pi:\WD_F\rightarrow {^LG}$ defined by
\begin{equation}
    \label{sigma to pi}\varphi_\pi(w,h):=\varphi_\sigma(w,1)\psi_{\sigma,\lambda}(h).
\end{equation}
This construction maps a discrete series representation of $G$ to a discrete admissible homomorphism $\varphi_{\pi}:\WD_F\rightarrow {^LG}$. Under some natural assumptions, the above construction could be extended to a well-defined map (but not uniquely determined in general)
$$\mathrm{LL}^{\bG}:\Pi(\bG)\rightarrow \Phi(\bG)$$
using \eqref{non-tempered parameters} and \eqref{tempered parameters}.

\begin{theorem}[Heiermann, \cite{heiermann-orbites}]\label{Heiermann construction of discrete parameters}
     Under a series of natural hypotheses of the LLC and L-packets for $\bG$ and its semi-standard Levi subgroups (for precise statements, cf. \cite[\textsection 6.5]{heiermann-orbites}), the map $\mathrm{LL}^{\bG}$ constructed above does induce a bijection between the set of L-packets of $\bG$ and $\Phi(\bG)$.
\end{theorem}
\begin{remark}The main input of the proof of this theorem is a characterization of the cuspidal support of discrete series via harmonic analysis. It claims that an irreducible supercuspidal representation of $M$ is the cuspidal support of some discrete series $\pi$ of $G$ if and only if the $\mu$-function has a pole of maximal order at $\sigma$ (cf. \cite{heiermann-spectrales}).

    For generic representations of a quasi-split group, the compatibility of certain local factors  associated to the above construction (namely $L(s,r_i\circ \varphi_{\pi})$, $\varepsilon(s,r_i\circ \varphi_{\pi},\psi)$, $\gamma(s,r_i\circ \varphi_{\pi},\psi)$ where $r_i$ as defined in \textsection \ref{subsection shahidi}) with Langlands-Shahidi local factors was studied in \cite{heiermann-unipotent}.
\end{remark}

This result to some extent characterizes what a ``reasonable" construction of discrete parameters (from supercuspidal ones) ``should" be, while the precise statements of those hypotheses in the theorem are quite long and technical. However, we shall not make any use of these precise statements in the sequel (thus omitted here). 

What we want to point out is: it is the properties $$\varphi_\pi(w,1)=\iota_{P}\circ\varphi_\sigma(w,1)\text{ for }w\in \W_F\quad \text{and}\quad\varphi_\pi(\begin{bmatrix}
    q^{1/2} & \\
    & q^{-1/2}
\end{bmatrix})=s_{\sigma,\lambda}$$ of $\varphi_{\pi}$ that will serve as the input for our computation of the adjoint $\gamma$-factors. Thus in the main theorem, this will be one of our ``working hypotheses" for LLC (which is not accessible for general groups).

\subsection{The Adjoint $\gamma$-Factors}\label{subsection adjoint gamma factor} Now we go back to the setting at the beginning of \textsection4, so that $\bG$ is a quasi-split group with associated data as explained before. We set
\begin{itemize}
    \item $\mathbf{P}=\mathbf{M}\mathbf{N}$ be the standard maximal parabolic subgroup corresponding to $\Delta-\{\alpha\}$;
    \item $\sigma$ an irreducible unitary generic supercuspidal representation of $M$;
    \item $\pi$ a discrete series representation of $G$ which is a subrepresentation of $i_P^G\sigma_\lambda$ for $\lambda=s_0\Tilde{\alpha}$, $s_0\in \{1,1/2\}$ (cf. \S\ref{subsection shahidi}); here we write $s_0$ rather than $1/j$ to make the ideas and details below clearer (since the exact number $1/j$ is not crucial here);
    \item $\varphi_{\sigma}$ (resp. $\varphi_{\pi}$) the L-parameter of $\sigma$ (resp. $\pi$), assuming the LLC.
\end{itemize} 
Of course, the LLC is inaccessible in general, so we need some ``working hypotheses" in such a general setting. It is a well-known folklore conjecture that generic supercuspidal representations should correspond to $\SL_2$-trivial discrete parameters (cf. \cite[pp. 103]{heiermann-unipotent} or more formally, Conjecture 7.1 of \cite{gross-reeder}). Thus we assume further that 
\begin{itemize}
\item $\varphi_{\sigma}|_{\SL_2(\CC)}$ is trivial, 
    \item $$\varphi_\pi(w,1)=\iota_{P}\circ\varphi_\sigma(w,1)\text{ for }w\in \W_F\quad \text{and}\quad\varphi_\pi(\begin{bmatrix}
    q^{1/2} & \\
    & q^{-1/2}
\end{bmatrix})=s_{\sigma,\lambda},$$
where $s_{\sigma,\lambda}$ is defined in \eqref{s-sigma-lambda}.
\end{itemize}

With these working hypotheses for LLC, we could now compute the adjoint $\gamma$-factor $\gamma(s,\pi,\Ad,\psi)$. Recall that $\W_F$ is equipped with an absolute value $\|\cdot \|$ (cf. \S\ref{subsection parameters}). For our computations, we need the following lemma, which is well-known to experts:
\begin{lemma}
    Let $V$ be a finite dimensional vector space over $\CC$ and $\varphi:\WD_F\rightarrow \GL(V)$ be an admissible homomorphism. Define 
    $\varphi':\WD_F\rightarrow \GL(V)$ by
    $$\varphi'(w,h):=\varphi(w,\begin{bmatrix}
        \|w\|^{1/2} &\\
        & \|w\|^{-1/2}
    \end{bmatrix}).$$
    Then 
    $$\gamma(s,\varphi,\psi)=\gamma(s,\varphi',\psi).$$
    \end{lemma}

    For a proof, see \cite[\textsection 3.2]{heiermann-unipotent}, and \cite[\textsection 3.4]{shahidi-plancherel}. In particular, we have the following

   \begin{corollary}
      Let $\pi$ be an irreducible smooth representation of $G$ with L-parameter $\varphi_\pi:\WD_F\rightarrow {^LG}$. Define $\varphi_\pi':\WD_F\rightarrow {^LG}$ 
by
$$\varphi_\pi'(w,h):=\varphi(w,\begin{bmatrix}
        \|w\|^{1/2} &\\
        & \|w\|^{-1/2}
    \end{bmatrix}),$$
    then for any finite dimensional  representation $r$ of $^LG$,
$$\gamma(s,\pi,r,\psi)=\gamma(s,r\circ\varphi_\pi',\psi).$$
   \end{corollary}

Now let $\widehat{\lieg}$ be the Lie algebra of $\widehat{G}$ and $\Ad$ be the adjoint representation of $^LG$ on $\widehat{\lieg}_{\ad}:=\widehat{\mathfrak{g}}/Z(\widehat{\mathfrak{g}})^\Gamma$ (not $\widehat{\lieg}/Z(\widehat{\lieg})$). We consider the composite
$$\W_F\xrightarrow{\varphi_\pi'}{^LG}\xrightarrow{\Ad}\GL(\widehat{\lieg}_{\ad}).$$
Since
\begin{align*}   \varphi_\pi'(w):=&\varphi_\pi(w,\begin{bmatrix}
        \|w\|^{1/2} &\\
        & \|w\|^{-1/2}
    \end{bmatrix} )\\
    =&\varphi_\sigma(w)\varphi_\pi(\begin{bmatrix}
        q^{1/2}&\\
        & q^{-1/2}
    \end{bmatrix})^{-v(w)}\\
    =&\varphi_{\sigma}(w)s_{\sigma,\lambda}^{-v(w)}\in {^L\!M}\subset {^LG},
\end{align*}
we first decompose the restriction of $\Ad$ onto ${^L\!M}$. Since $\widehat{\mathfrak{g}}=\widehat{\liem}\oplus \widehat{\lien}\oplus \widehat{ \overline{\lien}}$ and $Z(\widehat{\mathfrak{g}})^\Gamma\subset Z(\widehat{\liem})^\Gamma$, the adjoint action of ${^L\!M}$ on $\widehat{\lieg}_{\ad}$ can be decomposed as 
$$\widehat{\lieg}_{\ad}=\widehat{\liem}/Z(\widehat{\lieg})^\Gamma\oplus\widehat{\lien}\oplus\widehat{\overline{ \lien}}$$
as an $^L\!M$-module. Since $\mathbf{P}=\mathbf{M}\mathbf{N}$ is a standard maximal parabolic subgroup, ${^L\!P}$ is relevant, maximal,  $Z(\widehat{\liem})^\Gamma/Z(\widehat{\mathfrak{g}})^{\Gamma}$ is then one dimensional and isomorphic to the trivial representation $\mathbbm{1}$ of $^L\!M$. 

Recall by the Langlands-Shahidi theory, the adjoint representation of $^L\!M$ on $\widehat{\lien}$, denoted by $r$, can be decomposed as $(r,\widehat{\lien})=\bigoplus_{i=1}^m(r_i,V_i)$, where
\[V_i:=\{X_{\beta^\vee}\in \widehat{\lien}\mid \beta\in \Sigma, \langle \Tilde{\alpha},\beta^\vee\rangle=i\},\quad 1\leq i\leq m,\]
and each $(r_i,V_i)$ is irreducible. Similarly, $(\Tilde{r},\widehat{\overline{\lien}})=\bigoplus_{i=1}^m(\Tilde{r}_i,\overline{V}_i)$, where $\Tilde{r}$ (resp. $\Tilde{r}_i$) is the contragredient of $r$ (resp. $r_i$), and 
\[\overline{V}_i:=\{X_{\beta^\vee}\in \widehat{\overline{\lien}}\mid \beta\in \Sigma, \langle \Tilde{\alpha},\beta^\vee\rangle=(-i)\},\quad 1\leq i\leq m.\]
(Recall that $\Tilde{\alpha}:=\langle \rho_{\mathbf{P}},\alpha^\vee\rangle^{-1}\rho_{\mathbf{P}}\in\fa_{\mathbf{M}}^*$.)

So as a representation of $\W_F$, 
$$\Ad\circ \varphi_\pi'\isom \mathbbm{1}\oplus (\Ad_M\circ \varphi_\pi')\oplus \bigoplus_{i=1}^m[(r_i\circ \varphi_\pi')\oplus (\Tilde{r}_i\circ \varphi_\pi')],$$
where we denote by $\Ad_M$ the adjoint representation of ${^L\!M}$ on $\widehat{\liem}/Z(\widehat{\liem})^\Gamma$. By definition of $s_\lambda\in T_{^L\!M}$, it corresponds to $\lambda=-s_0\Tilde{\alpha}\in \fa_{\mathbf{M}}^*$ by the LLC for tori (cf. \S\ref{subsection parameters}). By our assumption $$s_\sigma=\varphi_{\sigma}(1,\begin{bmatrix}
    q^{1/2} & \\
    & q^{-1/2}
\end{bmatrix})\in \varphi_{\sigma}(\SL_2(\CC))$$ is trivial, thus $s_{\sigma,\lambda}:=s_{\lambda}s_{\sigma}=s_{\lambda}\in Z( {^L\!M})$ acts trivially on $\widehat{\liem}$, $\Ad_M\circ \varphi_\pi'\isom \Ad_M\circ \varphi_\sigma$ and
$$\gamma(s,\Ad_M\circ\varphi_\pi',\psi)=\gamma(s,\Ad_M\circ \varphi_\sigma,\psi)=\gamma(s,\sigma,\Ad,\psi).$$
(As we can see, here we need the $\SL_2$-trivial assumption for a generic supercuspidal parameter, otherwise the decomposition would be much more delicate.)

Since
$$V_i=\{X_{\beta^\vee}\in \widehat{\lien}\mid \beta\in\Sigma, \langle \Tilde{\alpha},\beta^\vee\rangle=i\},$$
we have
$$\Ad(s_\lambda)(v_i)=q^{-\langle \lambda,\beta^\vee\rangle }v_i=q^{is_0}v_i,\quad (v_i\in V_i),$$
thus
$$r_i\circ \varphi_\pi'\isom (r_i\circ \varphi_\sigma)\otimes \chi_i,$$
where $\chi_i:\W_F \rightarrow \CC^\times$ is the unramified character defined by $\chi_i(\Frob)=q^{is_0}$, and then
$$\gamma(s,r_i\circ \varphi_\pi',\psi)=\gamma(s+is_0,r_i\circ\varphi_\sigma,\psi)=\gamma(s+is_0,\sigma,r_i,\psi).$$
Similarly, since $\gamma(s,\sigma,\Tilde{r}_i,\psi)=\gamma(s,\Tilde{\sigma},r_i,\psi)$, we get
$$\gamma(s,\Tilde{r}_i\circ \varphi_\pi',\psi)=\gamma(s-is_0,\sigma,\Tilde{r}_i,\psi)=\gamma(s-is_0,\Tilde{\sigma},r_i,\psi).$$
Finally, 
$$\gamma(s,\mathbbm{1},\psi)=\frac{1-q^{-s}}{1-q^{s-1}}.$$

From these computations, we obtain the following
\begin{theorem} Following the setting at the beginning of this subsection,
    $$\gamma(s,\pi,\Ad,\psi)=\left(\frac{1-q^{-s}}{1-q^{s-1}}\right)\cdot \gamma(s,\sigma,\Ad,\psi)\cdot \prod_{i=1}^m\gamma(s+is_0,\sigma,r_i,\psi)\gamma(s-is_0,\Tilde{\sigma},r_i,\psi).$$
\end{theorem}
As explained in Lemma 1.2 of \cite{hiraga-ichino-ikeda}, $\gamma(0,\pi,\Ad,\psi),\gamma(0,\sigma,\Ad,\psi)$ are holomorphic and nonzero at $s=0$. Since when $s\rightarrow 0$
$$\frac{1-q^{-s}}{1-q^{s-1}}\sim \frac{s\cdot \log q}{1-q^{-1}},$$  we have
\begin{corollary}\label{adjoint gamma factor quotient} Following the setting at the beginning of this subsection,
    $$\frac{\gamma(0,\pi,\Ad,\psi)}{\gamma(0,\sigma,\Ad,\psi)}=\left(\frac{\log q}{1-q^{-1}}\right)\cdot \underset{s=0}{\Res}\prod_{i=1}^m\gamma(s+is_0,\sigma,r_i,\psi)\gamma(s-is_0,\Tilde{\sigma},r_i,\psi).$$
\end{corollary}

\section{The Conjecture and the Main Theorem}\label{section the main theorem}
In this section, we recall some details of the formal degree conjecture and prove our main result.
\subsection{The Formal Degree Conjecture}\label{subsection the conjecture} For simplicity, we shall still restrict ourselves to quasi-split groups. Thus $\bG$ is still a quasi-split connected reductive group with $\bA$ the split component of its center, and $\bB=\bT\bU$ a fixed Borel subgroup. Let $(\pi,V_\pi)$ be a discrete series representation of $G$ and we fix a Haar measure $\mu$. Recall that we denote by $d(\pi)=d(\pi,\mu)$ the formal degree of $\pi$ with respect to $\mu$ (cf. \S\ref{section measures}). 

Roughly speaking, discrete series representations of a real or $p$-adic reductive group are those irreducible unitary representations ``appearing discretely" in the decomposition of the regular representation on $L^2(G)$. The formal degree of a discrete series is exactly Harish-Chandra's Plancherel measure in his famous Plancherel formula. The explicit determination of Harish-Chandra's Plancherel measure has numerous local and global applications. 

The explicit Plancherel formula for a general real reductive group has been completed by Harish-Chandra in 1970s, which was one of the highest and most fundamental achievements in representation theory of Lie groups. This was based on his explicit classification of the discrete series for all real semi-simple groups. However, such an explicit classification is inaccessible for general $p$-adic reductive groups due to the enormous arithmetic (Galois) complexity. In \cite{hiraga-ichino-ikeda}, an explicit formula of $d(\pi,\mu)$ is conjectured in the framework of LLC. To state this conjecture in its precise form, we recall some desiderata of the refined LLC for quasi-split groups.

 Let ${^LG}=\widehat{G}\rtimes \Gamma$ be the L-group of $\bG$ and $\widehat{G}^\natural$ be the dual group of $\bG/\bA$. This is a subgroup of $\widehat{G}$ containing $\widehat{G}_{\der}$. Let $\varphi:\WD_F\rightarrow {^LG}$ be an L-parameter of $\bG$ with the L-packet $\Pi_{\varphi}$. Let
\begin{align*}
    &S_{\varphi}:=Z_{\widehat{G}}(\varphi(\WD_F)),~\overline{S}_{\varphi}:=S_{\varphi}/Z(\widehat{G})^{\Gamma},\\
    &S_{\varphi}^\natural:=Z_{\widehat{G}^\natural}(\varphi(\WD_F)),~\cS_{\varphi}^\natural:=\pi_0(S_{\varphi}^\natural).
\end{align*}
The ``usual" component group used in the refined LLC for quasi-split groups is $\pi_0(\overline{S}_\varphi)$, and $\cS_{\varphi}^\natural$ is the component group used in the formal degree conjecture. Here we should remark that since we only focus on quasi-split $\bG$, the component group $\pi_0(\overline{S}_\varphi)$ would suffice. For general connected reductive groups, let $\widehat{G}_{\mathrm{sc}}$ be the simply connected cover of the adjoint group $\widehat{G}_{\ad}:=\widehat{G}/Z(\widehat{G})$, then one has to use the pullback of $\overline{S}_\varphi$ in $\widehat{G}_{\mathrm{sc}}$ to define the component group $\cS_{\varphi}$ (see \cite[pp. 287]{hiraga-ichino-ikeda}).

We denote by $\Irr(K)$ the set of irreducible complex representations of a finite group $K$.

\begin{conjecture}[Refined LLC for $p$-adic Quasi-Split Groups]\label{refined LLC}Setting and notations as above.
    \begin{enumerate}
        \item For each L-parameter $\varphi$, there should exist a bijection $\Pi_{\varphi}(\bG)\rightarrow \Irr(\pi_0(\overline{S}_\varphi))$.
        \item (Shahidi) Given a Whittaker datum $\fw$, there should exist a unique $\fw$-representation in each tempered L-packet. 
        \item For tempered $\varphi$, there should exist a bijection $\iota_{\fw}:\Pi_{\varphi}(\bG)\rightarrow \Irr(\pi_0(\overline{S}_\varphi))$ (depending on $\fw$), which maps the unique $\fw$-generic representation to the trivial representation of $\pi_0(\overline{S}_\varphi)$.
        \item For tempered $\varphi$ and $\pi\in \Pi_{\varphi}$, the above bijection $\iota_{\fw}$ should be characterized by certain character identities, and $\dim (\iota_{\fw}(\pi))$ should be independent of the choice of $\fw$.
    \end{enumerate}
\end{conjecture}

Assuming the refined LLC, we can now state the formal degree conjecture. Let $\varphi$ be a discrete parameter of $\bG$, $\pi\in \Pi_{\varphi}$, and fix a bijection 
\[
\Pi_{\varphi}\rightarrow \Irr(\pi_0(\overline{S}_\varphi)),~\pi'\mapsto \rho_{\pi'}
\]
as explained above. Recall that we've fixed a nontrivial character $\psi:F\rightarrow \CC^1$. Let $\mu:=\mu_{2,\bG}$ be the Haar measure on $G/A$ as explained at the beginning of \textsection\ref{section measures}, which only depends on the choice of $\psi$. Denote by $\Ad$ the adjoint representation of ${^LG}$ on $\widehat{\lieg}/Z(\widehat{\lieg})^{\Gamma})$. Then the formal degree conjecture reads:
\begin{conjecture}[\cite{hiraga-ichino-ikeda}\cite{hiraga-ichino-ikeda-correction}]\label{formal degree conjecture}
    \[
    d(\pi)=\frac{\dim \rho_{\pi}}{|\cS_{\varphi}^\natural|}\cdot |\gamma(0,\pi,\Ad,\psi)|.
    \]
\end{conjecture}

Although we just stated the conjecture for quasi-split $\bG$, there is a precise version for general $\bG$, and it has been verified in various cases using different methods, for example (by no means complete):
\begin{itemize}
    \item Inner forms of $\GL(n)$:  Silberger-Zink, cf. \cite[\S 4]{hiraga-ichino-ikeda};
    \item Unitary groups $\mathrm{U}(n)$: Beuzart-Plessis \cite{bp-unitary} and Morimoto \cite{morimoto};
    \item Odd orthogonal groups $\SO(2n+1)$ and the metaplectic group $\mathrm{Mp}(2n)$: Ichino-Lapid-Mao \cite{ichino-lapid-mao};
    \item Kaletha's regular and non-singular supercuspidals of tame groups: Schwein \cite{schwein} and Ohara \cite{ohara};
    \item etc.
\end{itemize}
There was also an announcement by Beuzart-Plessis of a uniform proof for all classical groups via twisted endoscopy \cite{oberwolfach}.

By general representation theory of $p$-adic groups, a discrete series $\pi$ of $G$ is parabolically induced from a supercuspidal $\sigma$ of some Levi subgroup of $\bG$. Thus a natural idea is trying to reduce the conjecture to the supercuspidal case through the computation of $d(\pi)/d(\sigma)$. Of course this is quite difficult in general, since the cuspidal supports and L-parameters of discrete series for $p$-adic groups are very subtle in general (cf. \S\ref{subsection construction of discrete parameters}). This paper is such a first attempt: we deal with the case where $\bG$ is unramified, $\pi$ is generic and the parabolic is maximal, using tools from local harmonic analysis. As an application, we shall verify the conjecture for discrete series of split $\mathrm{G}_2$ supported on maximal Levi subgroups.

\subsection{The Main Theorem}
The setting of our main theorem is as follows:
\begin{itemize}
    \item $\bG$ is unramified, i.e. quasi-split and splits over an unramified extension;
    \item $\bP=\bM\bN$ is the standard parabolic subgroup corresponding to $ \Delta-\{\alpha\}$ for a positive simple root $\alpha\in \Delta$;
    \item $\sigma$ is an irreducible unitary generic supercuspidal of $M$, such that $i_P^G\sigma_{\Tilde{\alpha}/j}$ contains a unique generic discrete series subrepresentation $\pi$ of $G$ for some $j\in\{1,2\}$.
\end{itemize}

Denote by $\varphi_{\pi}$ (resp. $\varphi_{\sigma}$) the L-parameter of $\pi$ (resp. $\sigma$). Since the LLC is conjectural in general, following \S\ref{subsection adjoint gamma factor}, we have to assume the refined LLC (Conjecture \ref{refined LLC}), with
\begin{itemize}
\item $\varphi_{\sigma}|_{\SL_2(\CC)}$ is trivial, 
$$\varphi_\pi(w,1)=\iota_{P}\circ\varphi_\sigma(w,1)\text{ for }w\in \W_F\quad \text{and}\quad\varphi_\pi(\begin{bmatrix}
    q^{1/2} & \\
    & q^{-1/2}
\end{bmatrix})=s_{\sigma,\lambda},$$ cf.  \eqref{s-sigma-lambda} and \S \ref{subsection adjoint gamma factor}.
\end{itemize}
Since we used the Langlands-Shahidi theory, we also need to assume
\begin{itemize}
    \item Shahidi's local factors $\gamma^{\Sh}(s,\sigma,r_i,\psi)$ do coincide with those from the LLC (cf. the Remark of Theorem \ref{shahidi main theorem}); 
\end{itemize}

\begin{theorem}\label{main theorem}
    Under the above setting, we have
    \[
    \frac{d(\pi)}{d(\sigma)}=j^{-1}\cdot \frac{m}{\langle \chi,\alpha^\vee\rangle} \cdot  \frac{|\gamma(0,\pi,\Ad,\psi)|}{|\gamma(0,\sigma,\Ad,\psi)|},
    \]
    where 
    \begin{itemize}
        \item $\chi\in X^*(\mathbf{M})^{\mathbf{G}}\isom \ZZ$ is the generator with $\langle \chi,\alpha^\vee\rangle>0$;
        \item $m$ is the index of $\res(X^*(\mathbf{M})^{\mathbf{G}})$ in $X^*(\mathbf{A}_{\mathbf{M}})^{\mathbf{G}}$.
    \end{itemize}
    
    In particular, this result is compatible with the formal degree conjecture if and only if
\begin{equation}\label{compatibility with the conjecture}
    \frac{|\mathcal{S}^\natural_{\varphi_\pi}|}{|\mathcal{S}^\natural_{\varphi_\sigma}|}=j\cdot \frac{\langle \chi,\alpha^\vee\rangle}{m}.
\end{equation}
\end{theorem}

\begin{proof}
    We choose $K_G$ to be hyperspecial. Then by Proposition \ref{measure quotient},
    \[
    \frac{d(\pi)}{d(\sigma)}=\frac{\deg(\pi)}{\deg(\sigma)}\cdot \gamma(\mathbf{G}/\mathbf{M})^{-1}\cdot (1-q^{-1})^{-1}.
    \]
   For clarity, we write $s_0=1/j$. Using local harmonic analysis, we got
    \[
    \frac{\deg(\pi)}{\deg(\sigma)}=\gamma(\mathbf{G}/\mathbf{M})\cdot \log q\cdot \frac{m}{\langle \chi,\alpha^\vee\rangle}\cdot \underset{s=s_0}{\Res}\mu(\sigma\otimes \chi_{s\Tilde{\alpha}})
    \]
    in Proposition \ref{Heiermann's formula}, and
    \[\mu(\sigma\otimes \chi_{s\Tilde{\alpha}})=\prod_{i=1}^m \gamma^{\Sh}(is,\sigma,r_i,\Bar{\psi})\gamma^{\Sh}(-is,\Tilde{\sigma},r_i,\psi),\]
    by Shahidi's Theorem \ref{shahidi main theorem}. We've assumed that Shahidi's local factors $\gamma^{\Sh}$ coincide with the usual ones from LLC. Thus we have
\[
\frac{d(\pi)}{d(\sigma)}=\left(\frac{\log q}{1-q^{-1}}\right)\cdot \frac{m}{\langle \chi,\alpha^\vee\rangle}\cdot  \underset{s=s_0}{\Res}\prod_{i=1}^m\gamma(is,\sigma,r_i,\Bar{\psi})\gamma(-is,\Tilde{\sigma},r_i,\psi),
\]
and the simple pole appears in the term $\gamma(js,\sigma,r_i,\Bar{\psi})$ (cf. Theorem \ref{shahidi classification}).

By Corollary \ref{adjoint gamma factor quotient}, 
    \[
    \frac{\gamma(0,\pi,\Ad,\psi)}{\gamma(0,\sigma,\Ad,\psi)}=\left(\frac{\log q}{1-q^{-1}}\right)\cdot \underset{s=0}{\Res}\prod_{i=1}^m\gamma(s+is_0,\sigma,r_i,\psi)\gamma(s-is_0,\Tilde{\sigma},r_i,\psi),
    \]
   where this time the simple pole appears in the term $\gamma(s+js_0,\sigma,r_i,\psi)=\gamma(s+1,\sigma,r_i,\psi)$.

   Since 
   \[
\gamma(s,\sigma,r_i,\Bar{\psi})=\pm  \gamma(s,\sigma,r_i,\psi)
   \]   
  and for a meromorphic function $f(z)$ with a pole $z_0$,
   \[
   \underset{z=\lambda^{-1}z_0}{\Res} f(\lambda z)=\lambda^{-1}\underset{z=z_0}{\Res} f(z),
   \]
   we obtain
    \[
    \frac{d(\pi)}{d(\sigma)}=j^{-1}\cdot \frac{m}{\langle \chi,\alpha^\vee\rangle}\cdot  \frac{|\gamma(0,\pi,\Ad,\psi)|}{|\gamma(0,\sigma,\Ad,\psi)|}.
    \]
\end{proof}

Here the structure constant $\langle \chi,\alpha^\vee\rangle/n$ comes from the modification of the component group in the formal degree conjecture. Let us just take a look at the simplest nontrivial example to see what happens.

\begin{example}
    Let $\bG=\GL_{2n}$, $\bM=\GL_n\times \GL_n$ be the usual standard Levi subgroup, and $\bP=\bM\bN$ be the corresponding standard upper triangular parabolic subgroup of $\bG$. By the Bernstein-Zelevinsky classification, pick an irreducible unitary supercuspidal $\tau$ of $\GL_n(F)$, write $\sigma:=\tau\boxtimes\tau$ as a representation of $M$, then there is a unique discrete series subrepresentation $\pi$ of 
\[
i_P^G(\tau|\mathrm{det}_n|^{1/2}\boxtimes \tau|\mathrm{det}_n|^{-1/2}).
\]

We use the standard diagonal torus $\bT$ and the standard coordinates $X^*(\bT)=\oplus_{i=1}^{2n}\ZZ e_i$. If we identify $X^*(\bM)$ with a subgroup of $X^*(\bT)$, then 
\[
X^*(\bM)=\ZZ\cdot (e_1+\cdots +e_n)\oplus \ZZ\cdot (e_{n+1}+\cdots +e_{2n})
\]
and 
\[
X^*(\bM)^{\bG}=\ZZ\cdot [(e_1+\cdots +e_n)-(e_{n+1}+\cdots +e_{2n})],
\]
with $m:=[X^*(\bA_{\bM})^{\bG}:\res(X^*(\bM)^{\bG})]=n$. 

Of course the generator of $X^*(\bM)^{\bG}$ is $\chi=(e_1+\cdots +e_n)-(e_{n+1}+\cdots +e_{2n})$ and the coroot corresponding to our maximal parabolic $\bP$ is $\alpha^\vee=e_{n}^*-e_{n+1}^*$. Now the fundamental weight $\Tilde{\alpha}$ should be the element in $X^*(\bM)^{\bG}\otimes \RR=\fa_{\bM}^{\bG*}$ characterized by $\langle \Tilde{\alpha},\alpha^\vee\rangle =1$, so that
\[
\Tilde{\alpha}=\frac{1}{2}\cdot [(e_1+\cdots +e_n)-(e_{n+1}+\cdots +e_{2n})]
\]
corresponding to $\det_n^{1/2}\boxtimes \det_n^{-1/2}$. So for this example $j=1$, and the structure constant above becomes
\[
\frac{\langle \chi,\alpha^\vee\rangle}{m}=\frac{2}{n}.
\]

This number comes from the modification of the ``usual" component group. For $\bH=\GL_n$, a discrete L-parameter is just an irreducible representation of $\varphi:\WD_F\rightarrow \GL_n(\CC)=\widehat{H}$, so that $Z_{\widehat{H}}(\varphi)=Z(\widehat{H})$ and $\pi_0(\overline{S}_{\varphi})=1$. But $$\widehat{H}^\natural:=(\bH/\bA_{\bH})^\wedge = \SL_n(\CC)\subset \GL_n(\CC)=\widehat{H},$$
which leads to $\cS^\natural_{\varphi}=Z_{\widehat{H}^\natural}(\varphi)\simeq \mu_n$, so that in the above example 
\[
\frac{|\cS_{\varphi_\pi}^\natural|}{|\cS_{\varphi_\sigma}^{\natural}|}=\frac{2n}{n\cdot n}=\frac{2}{n},
\]
which is exactly the constant $\langle \chi,\alpha^\vee\rangle/m$ above.
\end{example}

\section{An Application to $\rG_2$}\label{section G2}
In this section, as an application of the main theorem, we verify the formal degree conjecture for discrete series of split $\rG_2$ supported on a maximal Levi subgroup. Let $\mathbf{G}=\mathrm{G}_2$ be the split exceptional group of type $\mathrm{G}_2$ over $F$. We first recall the classification of discrete series representations of $G=\mathbf{G}(F)$ supported on a maximal Levi.

Fix a maximal torus $\mathbf{T}$ of $\mathbf{G}$ and a Borel subgroup $\mathbf{B}$ containing $\mathbf{T}$. Let $\Sigma=\Sigma(\mathbf{G},\mathbf{T})$ be the root system and $\Delta=\{\alpha,\beta\}$ the set of simple roots relative to $\mathbf{B}$, with $\alpha$ short and $\beta$ long. Then the positive roots can be written as
\[
\alpha,\beta,\alpha+\beta,2\alpha+\beta,3\alpha+\beta,3\alpha+2\beta,
\]
with corresponding coroots
\[
\alpha^\vee,\beta^\vee,\alpha^\vee+3\beta^\vee,2\alpha^\vee+3\beta^\vee,\alpha^\vee+\beta^\vee,\alpha^\vee+2\beta^\vee,
\]
and
\[
\langle \alpha,\beta^\vee\rangle=(-1),\langle \beta,\alpha^\vee\rangle=(-3).
\]

Let $\mathbf{P}_{\alpha}=\mathbf{M}_{\alpha}\mathbf{N}_{\alpha}$ (resp. $\mathbf{P}_{\beta}=\mathbf{M}_{\beta}\mathbf{N}_{\beta}$) be the maximal parabolic subgroup of $\mathbf{G}$ corresponding to $\{\alpha\}$ (resp. $\{\beta\}$), i.e. $\mathbf{N}_{\alpha}$ (resp. $\mathbf{N}_{\beta}$) contains the root subgroup of $\beta$ (resp. $\alpha$) and $\mathbf{M}_{\alpha}$ (resp. $\mathbf{M}_{\beta}$) contains the root subgroup of $\alpha$ (resp. $\beta$). 

For convenience of the reader, we list the related normalizations and the decomposition $r=\oplus_i r_i$ (the adjoint representation of ${^L\!M}$ on $\widehat{\mathfrak{n}}$) for these two cases as follows.

\begin{itemize}
   \item $\mathbf{P}_{\alpha}=\mathbf{M}_{\alpha}\mathbf{N}_{\alpha}$, 
    \begin{align*}
           & \rho_\alpha=\frac{1}{2}[\beta+(\alpha+\beta)+(2\alpha+\beta)+(3\alpha+\beta)+(3\alpha+2\beta)]=\frac{3}{2}(3\alpha+2\beta),\\
           &\Tilde{\beta}=\langle \rho_\alpha,\beta^\vee\rangle ^{-1}\rho_\alpha=(3\alpha+2\beta).
        \end{align*}
        We may fix an isomorphism $\mathbf{M}_{\alpha}\isom\GL_2$ so that the character $\det$ of $\GL_2$ corresponds to $3\alpha+2\beta=\Tilde{\beta}$ of $\mathbf{M}_{\alpha}$, and the modulus character restricted to $\mathbf{M}_{\alpha}$ is $\delta_{\alpha}=|\det|^3$.  
        
        In this case the representation $r$ of $^L\!M_\alpha$ on $\widehat{\lien}_\alpha$ decomposes as
    \begin{center}
            \begin{tabular}{|c|c|c|}
            \hline 
               $r_1=\std$  & $r_2=\det$ & $r_3=\std\otimes \det$ \\

               \hline
                $\beta^\vee,(3\alpha+\beta)^\vee$ & $(3\alpha+2\beta)^\vee$ & $(\alpha+\beta)^\vee,(2\alpha+\beta)^\vee$\\
                \hline
            \end{tabular}
        \end{center}
        \item 
$\mathbf{P}_{\beta}=\mathbf{M}_{\beta}\mathbf{N}_{\beta}$,
         \begin{align*}
           & \rho_{\beta}=\frac{1}{2}[\alpha+(\alpha+\beta)+(2\alpha+\beta)+(3\alpha+\beta)+(3\alpha+2\beta)]=\frac{5}{2}(2\alpha+\beta),\\
           &\Tilde{\alpha}=\langle \rho_\beta,\alpha^\vee\rangle ^{-1}\rho_\beta =(2\alpha+\beta).
        \end{align*}
        We may fix an isomorphism $\mathbf{M}_{\beta}\isom\GL_2$ so that the character $\det$ of $\GL_2$ corresponds to $2\alpha+\beta=\Tilde{\alpha}$ of $\mathbf{M}_{\beta}$, and the modulus character restricted to $\mathbf{M}_{\beta}$ is $\delta_{\beta}=|\det|^5$.
        
         In this case the representation $r$ of $^L\!M_\beta$ on $\widehat{\lien}_\beta$ decomposes as
        \begin{center}
            \begin{tabular}{|c|c|}
            \hline 
               $r_1=\Sym^3\otimes \det^{-1}$  & $r_2=\det$\\

               \hline
                $\alpha^\vee,(\alpha+\beta)^\vee,(3\alpha+\beta)^\vee,(3\alpha+2\beta)^\vee$ & $(2\alpha+\beta)^\vee$\\
                \hline
            \end{tabular}
        \end{center}

\end{itemize}
Note that in these normalizations, the character $\det$ of $\bM_\alpha$ (resp. $\bM_\beta$) corresponds exactly to the positive root perpendicular to $\alpha$ (resp. $\beta$).

The discrete series of $G$ has been classified by Mui\'c. Here we just summarize below the cuspidal support of those supported on a maximal parabolic; see the original paper \cite{muic-unitary-dual-of-g2} of Mui\'c, and another convenient reference is \cite[\textsection 2]{gan-savin-g2theta}.

\begin{theorem}
    Let $\tau$ be an irreducible unitary supercuspidal representation of $\GL_2(F)$ which is self-dual, with central character $\omega_\tau$.
    \begin{enumerate}
    \item If $\omega_\tau=1$, then $i_{P_\alpha}^G(\tau\otimes |\det|^{1/2})$ has a unique irreducible subrepresentation $\pi_\alpha(\tau,1/2)$.
        \item If $\omega_\tau=1$, then $i_{P_\beta}^G(\tau\otimes |\det|^{1/2})$ has a unique irreducible subrepresentation $\pi_{\beta}(\tau,1/2)$.
         \item If $\omega_\tau\neq 1$, then $i_{P_\beta}^G(\tau\otimes |\det|)$ has a unique irreducible subrepresentation $\pi_{\beta}(\tau,1)$.
    \end{enumerate}
    The representations $\pi_{\alpha}(\tau,1/2),,\pi_{\beta}(\tau,1/2),\pi_{\beta}(\tau,1)$ are generic discrete series, and this classifies all the discrete series representations of $G$ supported on a maximal Levi.
\end{theorem}

The full explicit local Langlands correspondence for $\mathrm{G}_2$ has been established in \cite{gan-savin-g2-llc} based on a series work in exceptional theta correspondence. The conjectural L-parameters for non-supercuspidal discretes series have already been classified in \cite[\textsection 3.5]{gan-savin-g2theta}, which we now recall. 

We identify the $L$-group of (split) $\mathrm{G}_2$ with the complex simple Lie group $\mathrm{G}_2(\CC)$ with trivial center; also, we identify the roots of $\rG_2(\CC)$ with the coroots of $\bG=\rG_2$. Given a root $\gamma^\vee\in \Sigma^\vee$, we denote by $\SL_{2,\gamma^\vee}$ the $\SL_2$ generated by the root subgroups of $\pm \gamma^\vee$. Now on the dual side $\alpha^\vee$ (resp. $\beta^\vee$) is the long (resp. short) root, and 
\[
Z_{\rG_2}(\SL_{2,\alpha^\vee})=\SL_{2,(\alpha^\vee+2\beta^\vee)},~Z_{\rG_2}(\SL_{2,\beta^\vee})=\SL_{2,(2\alpha^\vee+3\beta^\vee)}.
\]
Here $(\alpha^\vee+2\beta^\vee)$ (resp. $(2\alpha^\vee+3\beta^\vee)$) is also the positive root perpendicular to $\alpha^\vee$ (resp. $\beta^\vee$). Finally, the long root subgroups of $\mathrm{G}_2(\CC)$ generate an $\mathrm{SL}_3$. 

Now let $\pi$ be an irreducible discrete series representation of $G$ with L-parameter $\varphi=\varphi_\pi:\W_F\times \mathrm{SL}_2(\CC)\rightarrow \mathrm{G}_2(\CC)$. We then have 
\begin{center}
    \begin{tabular}{|c|c|c|c|}
    \hline
      $\pi$   & $\varphi(\mathrm{SL}_2)$ & $\varphi(\W_F)$ & $\pi_0(\overline{S}_\varphi)=\cS_{\varphi}^\natural$ \\
      \hline
       $\pi_{\alpha}(\tau,1/2)$  & $=\SL_{2,(\alpha^\vee+2\beta^\vee)}$ & $\subset \mathrm{SL}_{2,\alpha^\vee}\subset {^L\!M_\alpha}$ & $\ZZ/2\ZZ$\\

        $\pi_{\beta}(\tau,1/2)$  & $=\SL_{2,(2\alpha^\vee+3\beta^\vee)}$ & $\subset \mathrm{SL}_{2,\beta^\vee}\subset {^L\!M_\beta}$ & $\ZZ/2\ZZ$\\

         $\pi_{\beta}(\tau,1)$  & $=\SO_3\subset \mathrm{SL}_3$ & $=Z_{\mathrm{G}_2}(\SO_3)\isom S_3$ &  1\\
         & &$\subset \SL_{2,\beta^\vee}\subset {^L\!M_\beta}$ & \\
         \hline
    \end{tabular}
\end{center}
by \cite[\S 3.5]{gan-savin-g2theta}. (Note that the representations $\pi_{\alpha}(\tau,1/2)$, $\pi_{\beta}(\tau,1/2)$, $\pi_{\beta}(\tau,1)$ above are denoted by $\delta_P(\tau)$, $\delta_Q(\tau)$, $\pi_{\mathrm{gen}}[\tau]$, respectively, in \cite{gan-savin-g2theta}.)

 We now check that these parameters are compatible with the condition in \S\ref{subsection adjoint gamma factor}. Matching the normalizations at the beginning of this section, we fix $^L\!M_\alpha\isom \GL_{2,\alpha^\vee}$ (resp. $^L\!M_\beta\isom \GL_{2,\beta^\vee}$) such that the character $\det$ of $\GL_2$ corresponds to $(\alpha^\vee+2\beta^\vee)$ (resp. $(2\alpha^\vee+3\beta^\vee)$), which is the root perpendicular to $\alpha^\vee$ (resp. $\beta^\vee$).
\begin{itemize}
    \item For $\pi=\pi_\alpha(\tau,1/2)$, $j=2$, we have to check that under the above normalization, 
    \[
    \varphi_\pi(\begin{bmatrix}
        q^{1/2} & \\
        & q^{-1/2}
    \end{bmatrix})\in Z(^L\!M_\alpha)
    \]
    corresponds to 
    \[
    t=\begin{bmatrix}
        q^{1/2} & \\
        & q^{1/2}
    \end{bmatrix}\in Z(\GL_{2,\alpha^\vee}).
    \]
    Here $t\in {^L\!T}$ is characterized by
    \[
    \alpha^\vee(t)=1,~\det(t)=(\alpha^\vee+2\beta^\vee)(t)=q.
    \]
    Switching to $\varphi_{\pi}:\SL_2\xrightarrow{\sim}\SL_{2,(\alpha^\vee+2\beta^\vee)}$, this characterizes exactly the semi-simple class of 
    \[
    \varphi_\pi(\begin{bmatrix}
        q^{1/2} & \\
        & q^{-1/2}
    \end{bmatrix})\in \SL_{2,(\alpha^\vee+2\beta^\vee)}.
    \]
    The case $\pi_\beta(\tau,1/2)$ is dual and similar.
    \item For $\pi=\pi_\beta(\tau,1)$, $j=1$, we have to check that
    \[
    \varphi_\pi(\begin{bmatrix}
        q^{1/2} & \\
        & q^{-1/2}
    \end{bmatrix})\in Z(^L\!M_\beta)
    \]
    corresponds to 
    \[
    t=\begin{bmatrix}
        q& \\
        & q
    \end{bmatrix}\in Z(\GL_{2,\beta^\vee}).
    \]
    This time $t\in {^L\!T}$ is characterized by
    \[
    \beta^\vee(t)=1,~\det(t)=(2\alpha^\vee+3\beta^\vee)(t)=q^2.
    \]
    Switching to the $\SL_3$ generated by the long root subgroups, this means
    \[
    \alpha^\vee(t)=q,~(\alpha^\vee+3\beta^\vee)(t)=q,~(2\alpha^\vee+3\beta^\vee)(t)=q^2.
    \]
    Now $\varphi_\pi:\SL_2\rightarrow\SO_3\subset \SL_3$ is the three dimensional irreducible representation, so the above characterizes exactly the semi-simple class of 
    \[
    \varphi_{\pi}(\begin{bmatrix}
        q^{1/2} &\\
        & q^{-1/2}
    \end{bmatrix})\in \SL_3.
    \]
\end{itemize}

For the structure constants we have the following table:

\begin{center}
    \begin{tabular}{|c||c|c|c|c|c| c|}
    \hline
     $\bG$ & $\gamma$ & $\bM_{\gamma}$ & $m$ & $\chi\in X^*(\bM_{\gamma})^{\bG}$ generator & $\gamma^\vee$ & $\frac{m}{\langle \chi,\gamma^\vee\rangle}$\\

\hline
    
        $\rG_2$ & long $\beta$ & $\GL_2$ & $2$ & $\det=(3\alpha+2\beta)=\Tilde{\beta}$ & short $\beta^\vee$ & $2$ \\

          $\rG_2$ & short $\alpha$ & $\GL_2$ & $2$ & $\det=(2\alpha+\beta)=\Tilde{\alpha}$ & long $\alpha^\vee$ & $2$ \\
        \hline
       
    \end{tabular}
    \end{center}
    Since $\rG_2$ is semisimple and adjoint, $\cS_{\varphi_\pi}^\natural=\pi_0(\overline{S}_{\varphi_\pi})$, while for $\bM_\gamma\simeq \GL_2$ we have $\cS_{\varphi_\tau}^\natural=2$, so that the result of our main theorem reduces to 
    \[
    \frac{|\cS_{\varphi_\pi}^\natural|}{|\cS_{\varphi_\sigma}^\natural|}=j\cdot \frac{\langle \chi,\alpha^\vee\rangle}{m}\Longleftrightarrow |\pi_0(\overline{S}_{\varphi_\pi})|=j,
    \]
    which is clearly compatible with the above three classes of L-parameters.

The compatibility of Langlands-Shahidi $\gamma$-factors and those from LLC for split $\rG_2$ has been verified in \cite{shahidi-g2}. 
Of course, the formal degree conjecture is known for $\mathbf{M}_{\alpha},\mathbf{M}_{\beta}\isom \GL_2$; thus as an application we have verified the formal degree conjecture for discrete series of $\mathrm{G}_2$ supported on maximal Levi subgroups.

\subsection*{Acknowledgements} The author would like to first thank his advisor A. Ichino for suggesting this problem and many helpful discussions. He thanks professor W. T. Gan for pointing out a misunderstanding of the structure of L-parameters in \cite{gan-savin-g2theta}, and professor A.-M. Aubert and V. Heiermann for correcting some imprecision in \S\ref{subsection construction of discrete parameters} related to the construction in \cite{heiermann-orbites}.

The author also sincerely thanks professor A.-M. Aubert for her careful reading of the draft and bringing related sections of \cite{haines}, \cite{aubert-xu} to him.

\subsection*{Data Availability Statement}
The author confirms that all data generated or analyzed in this study are available within the article.

\printbibliography[heading=bibintoc,] 

\end{document}